\documentclass{amsart}
\usepackage{amssymb, a4wide, mathdots, url,hyperref,graphicx}
\usepackage[usenames]{color}
\usepackage{enumerate}
\usepackage{wasysym}
\usepackage{dsfont}
\usepackage{tikz-cd}
\usepackage{comment}
\usepackage{MnSymbol}
\usepackage{eqnarray, amsmath}
\usepackage{kotex}
\def\N{{\mathbb N}} \def\Z{{\mathbb Z}}

 \def\Q{{\mathbb Q}}

\def\R{{\mathbb R}} \def\C{{\mathbb C}}

\def\lcm{\operatorname{lcm}}

\let\caron\v

\def\k{{\bf k}}

\def\e{{\bf e}}

\def\v{{\bf v}}
\def\w{{\bf w}}

\def\c{{\bf c}}
\def\d{{\bf d}}

\newcommand{\unb}{\boldsymbol{1}}

\newtheorem{theorem}{Theorem}[section]
\newtheorem*{theorem*}{Theorem}
\newtheorem{lemma}[theorem]{Lemma}
\newtheorem{definition}[theorem]{Definition}
\newtheorem{proposition}[theorem]{Proposition}

\theoremstyle{remark}
\newtheorem{remark}[theorem]{Remark}
\newtheorem{hypothesis}{Hypothesis}

\title[Divisor moments of polynomials in modular forms]{Divisor moments of polynomials in Fourier coefficients of modular forms}

\author{Wonwoong Lee}
\address{Department of
 Mathematics Education, Chonnam National University, Yongbong-ro, Buk-gu, Gwangju 61186, Chonnam Province, South Korea}
\email{dldnjsdnd041@gmail.com}

\thanks{}


\begin{document}

\begin{abstract}
    We study higher moments of the divisor function evaluated at polynomial expressions in the Fourier coefficients of a non-CM newform. The logarithmic exponent appearing in our estimates depends only on the number of irreducible factors of the polynomial and remains unchanged under a Sato--Tate restriction. The proof combines an effective Chebotarev theorem, or an effective Chebotarev--Sato--Tate theorem, with the arithmetic of joint cycle types and a mean value estimation for multivariable multiplicative functions with Frobenian support.
\end{abstract}

\maketitle

\section{Introduction}

\subsection{Main result}

Let $d(n)$ denote the number of positive divisors of $|n|$ for a nonzero integer $n$.  Let $f$ be a non-CM newform of weight $k\geq 2$ and level $N$ with integral Fourier coefficients $a_f(n)$.  For every prime $p\nmid N$, write
\[
    \frac{a_f(p)}{p^{(k-1)/2}}=2\cos\theta_f(p), \quad \theta_f(p)\in [0,\pi].
\]
Given an interval $I\subseteq[0,\pi]$ and a squarefree polynomial $F\in\mathbb Z[T]$, we study the higher moments
\[
    d\left(F(a_f(p))\right)^m
\]
as $p$ ranges over primes satisfying $\theta_f(p)\in I$.

The main result of this paper is that the logarithmic order of these moments records the factorization of $F$. More precisely, if
\[
    F(T)=F_1(T)\cdots F_s(T)
\]
is a product of pairwise distinct irreducible polynomials over $\Q$, then the $m$th moment has order 
\[
    \asymp (\log x)^{s(2^m-1)}
\]
under the hypotheses stated below. Thus, the logarithmic exponent is determined by the number of irreducible factors of $F$. The same exponent persists when the primes are restricted to a fixed Sato--Tate interval $I$.

This problem continues a classical theme in the theory of the divisor function. Dirichlet's divisor problem gives
\[
    \frac{1}{x}\sum_{n\leq x}d(n)=\log x+(2\gamma-1)+O(x^{-1/2}),
\]
and Wilson \cite{Wil23} proved that, for every fixed $m\in\mathbb N$,
\[
    \frac{1}{x}\sum_{n\leq x}d(n)^m \sim A_m(\log x)^{2^m-1}
\]
with $A_m>0$. 
For an irreducible polynomial $F\in\mathbb Z[T]$, Erd\"os \cite{Erd52} proved
\[
    \frac{1}{x}\sum_{n\leq x}d(F(n)) \asymp_F \log x,
\]
and asymptotic formulae are known in several special cases, including irreducible quadratic polynomials \cite{Mc95,Mc99}. Delmer \cite{Del71} used Erd\"os' method to prove
\[
    \frac{1}{x}\sum_{n\leq x}d(F(n))^m \asymp_F (\log x)^{2^m-1}.
\]

A related family of questions arises when the averaging variable is restricted to the primes. The classical Titchmarsh divisor problem concerns $d(p-a)$ for a fixed nonzero integer $a$. Titchmarsh \cite{Tit30} proved the corresponding asymptotic under GRH, and Linnik \cite{Lin63} later proved it unconditionally. Geometric and
automorphic analogues were subsequently considered by Akbary and Ghioca \cite{AG12}, Pollack \cite{Pol16}, Gun and Murty \cite{GM14}, and Chiriac \cite{Chi22}.

For Fourier coefficients of modular forms, Gun and Murty proved, under GRH,
\[
    \log x \ll_f \mathbb E_{\Pi(x;f)}[d(a_f(X))] \ll_f (\log x)^{A_f},
\]
and conjectured an asymptotic of order $\log x$. Chiriac later
obtained the sharp conditional upper bound
\[
    \mathbb E_{\Pi(x;f)}[d(a_f(X))] \ll_f \log x.
\]
In previous work of Lau and the author \cite{LL26}, higher moments were studied after imposing a Sato--Tate restriction. If
\[
    \Pi(x,I;f):=\left\{p\leq x: p\nmid N,\  a_f(p)\neq0,\ \theta_f(p)\in I \right\},
\]
then for every fixed $m\in\mathbb N$,
\[
    \mathbb E_{\Pi(x,I;f)} \left[d(a_f(X))^m\right] \asymp_{f,m} (\log x)^{2^m-1}
\]
under the hypotheses stated in \cite{LL26}.

The passage from $a_f(p)$ to $F(a_f(p))$ introduces a new local arithmetic feature. The divisibility condition
\[
    \ell\mid F(a_f(p))
\]
is no longer a single trace-zero condition. It asks whether the residual trace $a_f(p)\bmod\ell$ belongs to the root set of $F\bmod\ell$.  Consequently, the relevant congruence densities depend on the reductions of the irreducible factors of $F$. Our main result determines how these local factorization data accumulate
in the global divisor moments.

For a finite set $M$ and a function $g$ on $M$, write
\[
    \mathbb E_M[g(X)]:= \frac{1}{\#M}\sum_{a\in M}g(a).
\]
For a polynomial $F$, define
\[
    \Pi(x,I;F,f):=\left\{p\leq x: p\nmid N,\ \theta_f(p)\in I,\ F(a_f(p))\neq 0 \right\}.
\]

\begin{theorem}\label{thm_main_polynomial_moments}
Let
\[
    F(T)=F_1(T)\cdots F_s(T)\in\mathbb Z[T],
\]
where $F_1,\ldots,F_s$ are pairwise distinct irreducible polynomials over $\mathbb Q$. Let $f$ be a non-CM newform of weight $k\geq2$ and level $N$ with integral Fourier coefficients, let $I\subseteq[0,\pi]$ be an interval of positive Sato--Tate measure, and let $m\geq1$.

Assume GRH. If $I\neq[0,\pi]$, assume in addition
Hypothesis~\ref{hypo_nice analytic properties of L-function}. Then
\[
    \mathbb E_{\Pi(x,I;F,f)}\left[d\left(F(a_f(X))\right)^m\right] \asymp_{f,F,m} (\log x)^{s(2^m-1)}
\]
as $x\to\infty$. The implied constants are independent of $I$.
\end{theorem}

\begin{remark}
Here, GRH refers to the generalized Riemann hypothesis for two types of $L$-functions. The first one is the Dedekind $L$-function, which is used to ensure the analytic continuation of the zeta function associated to cycle type (\S~\ref{sec_zeta function associated to cycle type}) and prove Lemma~\ref{lem_pi(x,[0,pi],F,f,delta) formula}. The second assumption concerns the following $L$-functions. For every natural number $\delta$, every finite Galois group  
$G = \mathrm{Gal}(K/\mathbb Q) \subseteq \mathrm{GL}_2(\mathbb Z/\delta)$, 
and every irreducible representation $\rho$ of $G$ with character $\chi$ (including the trivial character), we assume GRH for $L(s,\mathrm{Sym}^m f \otimes \chi)$.
\end{remark}

\begin{remark}
    When $F(T)=T$, Theorem~\ref{thm_main_polynomial_moments} recovers \cite[Theorem~1.1]{LL26}. 
\end{remark}

Here, Hypothesis~\ref{hypo_nice analytic properties of L-function} is as follows. Let $K/\mathbb Q$ be a finite Galois extension with Galois group $G$, and let $\chi$ denote a non-trivial irreducible character of $\mathrm{Gal}(K/\mathbb Q)$. The Langlands reciprocity conjecture predicts that there exists a cuspidal automorphic representation $\pi_{\chi}$ of $\mathrm{GL}_{\chi(1)}(\mathbb A_{\mathbb Q})$ associated with $\chi$. The following are the conjectural analytic properties of the Rankin-Selberg $L$-function $L(s,\mathrm{Sym}^m f \otimes \chi)$. We refer the reader to \cite[\S~5]{LL26} or \cite{ST19} for details of this $L$-function.

\begin{hypothesis}\label{hypo_nice analytic properties of L-function}
    Suppose $G$ is isomorphic to a subgroup of $\mathrm{GL}_2(\mathbb Z/\delta {\mathbb Z})$ for some $\delta \in \mathbb N$. Let $\chi$ be the irreducible character afforded by the representation $\rho_{\chi}$, where $\rho_{\chi}$ is a finite-dimensional complex representation of $\mathrm{Gal}(K/\mathbb Q)$. Then the $L$-function $L(s,\mathrm{Sym}^m f \otimes \chi)=L(s,\mathrm{Sym}^m \rho_f \otimes \rho_{\chi})$ satisfies the following properties.
    \begin{enumerate}
        \item The gamma factor of $L(s,\mathrm{Sym}^m f \otimes \chi)$ is given by
        \begin{align*}
            & L_{\infty}(s,\mathrm{Sym}^m f \otimes \chi)\\
            &=\begin{cases}
                \prod_{i=1}^{(m+1)/2}\Gamma_{\mathbb C}(s+(i-\frac{1}{2})(k-1))^{\chi(1)}, & \text{if $m$ is odd,} \\ \Gamma_{\mathbb R}(s+\epsilon)^{n^+(\chi)}\Gamma_{\mathbb R}(s+\epsilon')^{n^-(\chi)}\prod_{i=1}^{m/2}\Gamma_{\mathbb C}(s+i(k-1))^{\chi(1)}, &\text{if $m$ is even,}
            \end{cases}
        \end{align*}
        where $\Gamma_{\mathbb R}(s):=\pi^{-s/2}\Gamma(s/2)$, $\Gamma_{\mathbb C}(s):=\Gamma_{\mathbb R}(s)\Gamma_{\mathbb R}(s+1)$, and $\epsilon,\epsilon' \in \{0,1\}$, $\epsilon+\epsilon' \equiv 1 \pmod 2$, $\epsilon \equiv m/2 \pmod 2$.
        \item Define $\Lambda(s,\mathrm{Sym}^m f \otimes \chi):=L_{\infty}(s,\mathrm{Sym}^m f \otimes \chi)L(s,\mathrm{Sym}^m f \otimes \chi)$ and let $\delta(m,\chi)=1$ if $m=0$ and $\chi$ is trivial, and $\delta(m,\chi)=0$ otherwise. Then the function $$s^{\delta(m,\chi)}(1-s)^{\delta(m,\chi)}\Lambda(s,\mathrm{Sym}^m f \otimes \chi)$$ extends to an entire function on $\mathbb C$ of order $1$ which does not vanish at $s=0,1$.
        \item The function $\Lambda(s,\mathrm{Sym}^m f \otimes \chi)$ satisfies the functional equation
        \begin{align*}
             & \Lambda(s,\mathrm{Sym}^m f \otimes \chi)=\epsilon(\mathrm{Sym}^m f \otimes \chi)q(\mathrm{Sym}^m f\otimes \chi)^{1/2-s}\Lambda(1-s,\mathrm{Sym}^m f \otimes \overline{\chi}),
        \end{align*}
        where $\epsilon(\mathrm{Sym}^m f \otimes \chi)$ is a complex number of absolute value $1$ (known as the root number) and $q(\mathrm{Sym}^m f \otimes \chi)$ is the (arithmetic) conductor of $\mathrm{Sym}^m f \otimes \chi$.
        \item The analytic conductor $\mathfrak q(\mathrm{Sym}^m f \otimes \chi)$ satisfies the inequality
        \begin{align*}
            \mathfrak q(\mathrm{Sym}^m f\otimes \chi) \leq \mathfrak q(\mathrm{Sym}^m f)^{\chi(1)}\mathfrak q(\chi)^{m+1}. 
        \end{align*}
    \end{enumerate}
\end{hypothesis}

\subsection{Divisor moment estimate} 

We state an auxiliary theorem that generalizes the corresponding result of \cite{LL26}. The statement is formulated with sufficient flexibility to accommodate polynomial values with several irreducible factors.

Let $q$ be a fixed rational prime. Denote by $\mathbb N_q$ a multiplicative semigroup generated by a prescribed set of primes larger than $q$. Let $\{A_x\}_{x\in (0,\infty)}$ be a collection of finite sets such that $A_i \subset A_j$ for $i\le j$ and $A:= \bigcup_x A_x$ is an infinite set. For $n:A\to \mathbb N$ a function and an integer $\delta \ge 1$, let
    \begin{align*}
        \pi(x,\delta):=\#\{a \in A_x: n(a) \equiv 0 \text{ (mod } \delta)\}.
    \end{align*}
Let $\psi: \N \to \mathbb R_{>0}$ be a function satisfying the following conditions: there exists a function $\mathcal G:\N \to \R_{>0}$

\begin{itemize}
    \item (U2-1) For any $x \in \mathbb R_{>0}$,
    \begin{align*}
        \sum_{\substack{d_1, \ldots, d_m \leq x \\ \delta:=\lcm(d_1,\ldots,d_m)}}\frac{1}{\psi(\delta)} \ll (\log x)^{s(2^m-1)}.
    \end{align*}
    \item (U2-2) For any $x \in \mathbb R_{>0}$,
    \begin{align*}
        \sum_{\substack{d_1, \ldots, d_m \leq x \\ \delta:=\lcm(d_1,\ldots,d_m)}} \frac{\delta}{\psi(\delta)^2 \mathcal G(\delta)} \ll (\log x)^{s(2^m-1)}.
    \end{align*}
    \item (U2-3) Let $c \in (0,1)$. For $x\geq 3$, if $4/c\leq u < \log x/(2\log\log x)$,
    \begin{align*}
        \sum_{\substack{d_1,\ldots,d_m \in S(x^c,x^{1/u}) \\ d_1,\ldots,d_m \geq x^{c/4} \\ \delta:=\lcm(d_1,\ldots,d_m)}} \frac{\mathcal G(\delta)}{\delta} \ll (\log x)^{s(2^m-1)}\exp\left(-\frac{cu\log u}{16}\right),
    \end{align*}
    where $S(X,Z)$ denotes the set of positive integers $\delta \leq X$ such that every prime factor of $\delta$ is smaller than or equal to $Z$. If $u\geq \frac{\log x}{2\log\log x}$, then
    \begin{align*}
    \sum_{\substack{d_1,\ldots,d_m\in S(x^c,x^{1/u})\\
    d_1,\ldots,d_m\geq x^{c/4}\\
    \delta:=\lcm(d_1,\ldots,d_m)}}
    \frac{\mathcal G(\delta)}{\delta} \ll x^{-\eta}
    \end{align*}
    for some $\eta=\eta(c)>0$.
    \item (U2-4)  If $\delta \geq 3$, 
    \begin{align*}
        \frac{\delta}{\psi(\delta)\mathcal G(\delta)} \ll (\log \log \delta)^{B}
    \end{align*}
    for some constant $B$.
\end{itemize}
We also let $\varrho :\mathbb N_q \to \mathbb R_{>0}$ be a function satisfying the following condition:
\begin{itemize}
    \item (L-1) For any $x \in \mathbb R_{>0}$,
    \begin{align*}
        \sum_{\substack{d_1, \ldots, d_m \leq x \\ \delta:=\lcm(d_1,\ldots,d_m) \in \mathbb N_q}}\frac{1}{\varrho(\delta)} \gg (\log x)^{s(2^m-1)}.
    \end{align*}
\end{itemize}
Now we consider the following conditions on $n(a)$ and $\pi(x,\delta)$.

\begin{itemize}
    \item (U1) For some $\beta>0$, $|n(a)| < x^{\beta}$ for any $a \in A_x$, $x>0$.
    \item (U2) There exists $c_0 \in (0,1)$ such that if $1 \leq \delta \leq x^{c_0}$, then 
    \begin{align*}
        \pi(x,\delta) \ll \frac{1}{\psi(\delta)}\# A_x.
    \end{align*}
    \item (L) There exists $c_1 \in (0,1)$ such that if $\delta \in \mathbb N_q$, $1\leq \delta \leq x^{c_1}$, then
    \begin{align*}
        \pi(x,\delta) \gg \frac{1}{\varrho(\delta)}\# A_x.
    \end{align*}
\end{itemize}

By following the argument in \cite{LL26}, we can prove the following two auxiliary results.

\begin{theorem}\label{thm_auxiliary upper}
    If $n(a)$ and $\pi(x,\delta)$ satisfy (U1) and (U2), then
    \begin{align*}
        \sum_{a \in A_x} d(n(a))^m \ll \# A_x (\log x)^{s(2^m-1)}
    \end{align*}
    for sufficiently large $x$.
\end{theorem}

\begin{theorem}\label{thm_auxiliary lower}
    If $\pi(x,\delta)$ satisfies (L), then
    \begin{align*}
        \sum_{a \in A_x} d(n(a))^m \gg \# A_x (\log x)^{s(2^m-1)}
    \end{align*}
    for sufficiently large $x$.
\end{theorem}

\subsection{Multivariable mean value estimation with prime class-dependent local data}

Another ingredient of independent interest is a mean value estimate for multivariable multiplicative functions whose local behavior varies among finitely many classes of primes. The result extends the framework of \cite{EST22}, in which the local behavior is determined by a single set of data. In the present problem, the relevant local factors depend on the joint splitting type of the irreducible factors of $F$, and hence different classes of primes carry different local data.

Suppose that
\[
    \mathbb P=\mathcal Q_1\sqcup\cdots\sqcup\mathcal Q_t
\]
is a finite partition of the primes and that, on each $\mathcal Q_i$, a nonnegative multivariable multiplicative function $\varphi$ is governed by local data
\[
    (g_i,\kappa,\mathbf c,\xi_i,\mathcal Q_i).
\]
The normalized contribution of the $i$th prime class to the order corresponding to $\nu$ is $\lambda_i(\nu)$. These individual quantities need not be integers. The essential condition is instead that their total contribution
\[
    \Lambda(\nu):=\sum_{i=1}^t\lambda_i(\nu)
\]
be a nonnegative integer. Thus, the possibly fractional order arising from the individual prime classes combine to form an integral order in the full Euler product.

Let
\[
    I:=\left\{
        \nu\in\mathbb N_0^m:
        \kappa(\nu)=1,\ \Lambda(\nu)>0
    \right\}.
\]
Under suitable analytic assumptions on the partial zeta functions associated with the sets $\mathcal Q_i$, together with the full-rank and interior-cone conditions on the combined set, we prove that
\[
    \sum_{d_1,\ldots,d_m\leq x}
    \varphi(d_1,\ldots,d_m)
    =
    C(\varphi)x^{|\mathbf c|_1}(\log x)^\rho
    +
    O\left(x^{|\c|_1}(\log x)^{\rho-1}\right),
\]
where $C(\varphi)>0$ and
\[
    \rho
    =
    \sum_{\nu\in I}\Lambda(\nu)
    -
    \#\left\{
        j\in\{1,\ldots,m\}:c_j\neq0
    \right\}.
\]
The precise statement is given in
Theorem~\ref{thm_multivariable tauberian with many local data}.

In our application, the sets $\mathcal Q_i$ are the Frobenian sets
$\mathcal Q_{\tau}$ determined by the joint cycle types of the irreducible factors of $F$. It will be turned out that
\[
    \rho=s(2^m-1).
\]
This explains the occurrence of the logarithmic exponent in
Theorem~\ref{thm_main_polynomial_moments} and, in particular, why it depends only on the number of irreducible factors of $F$.

\subsection{Strategy of the proof}

The proof separates into an arithmetic distribution problem and an
analytic estimation. The arithmetic part determines the density with which
\[
    F(a_f(p))\equiv0\pmod \ell
\]
occurs among primes satisfying $\theta_f(p)\in I$.  The analytic part
converts these densities into the logarithmic order of the
divisor moment.

To be more precise, let us briefly describe how the auxiliary divisor moment estimate is applied to the sequence of polynomial values $F(a_f(p))$. The first step is to verify the distributional hypotheses appearing in the auxiliary theorem. For this purpose, we use the effective Chebotarev or Chebotarev--Sato--Tate theorem to count primes subject simultaneously to the relevant congruence conditions and to the restriction $\theta_f(p)\in I$. The density function entering this counting result, analogous to the function $h_f$ in \cite{LL26}, is determined by the adelic open image theorem for the Galois representation attached to $f$ and by some group-theoretic arguments.

After these distributional inputs have been established, the remaining task is to bound a reciprocal sum associated with arithmetic functions $\psi_F$ and $\varrho_F$ defined in terms of $F$. These functions are introduced in \S~\ref{sec_bounds for the density function}. To do this, we use properties of partial zeta functions attached to Frobenian sets in the sense of Serre \cite{Ser75}, together with the mean value estimate for multivariable multiplicative functions with prime class-dependent local data in Section~\ref{sec_tauberian}. We note that, in the earlier setting of \cite{LL26}, the theorem from \cite{EST22} was sufficient; here the presence of the polynomial $F$ imposes additional splitting conditions on the relevant primes, making the refinement necessary.

Some technical arguments are required to control the simultaneous contributions of the different factors when $F$ has more than one irreducible factor. We consider the zeta function associated to cycle type in \S~\ref{sec_zeta function associated to cycle type}. This zeta function admits an analytic continuation under GRH for the Dedekind zeta functions (see Proposition~\ref{prop_zeta function associated to cycle type is expressed as Dedekind zeta function}.)

The proof of our main result is summarized in \S~\ref{sec_proof of main thm}.

\subsection{Notational remark}

We use both notations $\# A$ and $|A|$ to denote the number of elements of $A$. The former is usually used in set-theoretic situation, whereas the latter is used in group-theoretic contexts. We also note that $p$ and $\ell$ always denote rational primes.

\subsection*{Acknowledgement}

The author would like to express sincere gratitude to Yuk-Kam Lau for numerous insightful discussions and valuable suggestions for this work. The author is also grateful to M. Ram Murty for suggesting that the author revisit Erd\"os's paper \cite{Erd52}, which helped remove some assumptions imposed in an earlier version of the paper. The author also thanks Liubomir Chiriac for a stimulating comment. The author used OpenAI's ChatGPT (GPT-5.6 Thinking, accessed in July 2026) to assist with language editing, the organization of parts of the exposition, and preliminary consistency checks. All mathematical statements and arguments were independently established, reviewed, and verified by the author who takes full responsibility for the content of the paper. The author has been supported by the National Research Foundation of Korea
(NRF) grant funded by the Korea government (MSIP) [No. RS-2024-00341327 and No. RS-2024-00415601
(G-BRL)].

\section{Proof of Divisor moment estimate}

In this section, we prove Theorems~\ref{thm_auxiliary upper} and \ref{thm_auxiliary lower}. The argument is essentially similar to that in \cite[\S~4]{LL26}. We provide the details for completeness.

\subsection{Upper bound (Proof of Theorem~\ref{thm_auxiliary upper})}

First of all, if we collect $a\in A_x$ for which $n(a)=1$, then 
\begin{align*}
    \sum_{\substack{a \in A_x \\ n(a)=1}} d(n(a))^m=\# A_x \ll \# A_x(\log x)^{s(2^m-1)}.
\end{align*}
Henceforth, we may assume $n(a) > 1$.

Set $c:=c_0/m.$
Fix $a\in A_x$ and write the prime factorization of $n(a)$ as
\begin{align*}
n(a)=p_1\cdots p_J, \qquad p_1\leq\cdots\leq p_J.
\end{align*}
For $0\leq i\leq J$, put
\begin{align*}
P_i:=p_1\cdots p_i, \qquad P_0:=1,
\end{align*}
and let $j$ be the largest index for which
\begin{align*}
P_j\leq x^c.
\end{align*}
The quantities $J$ and $j$ may depend on $a$.

We first separate the terms for which $j=0$:
\begin{align*}
\sum_{a\in A_x}d(n(a))^m = \sum_{\substack{a\in A_x \\ j=0}}d(n(a))^m + \sum_{\substack{a\in A_x\\ j\geq1}}d(n(a))^m 
=:\Sigma_0+\Sigma_+.
\end{align*}
Suppose that $j=0$. Every prime factor of $n(a)$ is larger than $x^c$. Hence
\begin{align*}
x^{cJ}<n(a)<x^\beta
\end{align*}
by (U1) so that $J<\beta/c$. Since $d(n(a))\leq2^J$, we have
\begin{align*}
d(n(a))^m\leq 2^{m\beta/c}.
\end{align*}
It follows that
\begin{align*}
\Sigma_0\ll \#A_x.
\end{align*}

We now deal with $\Sigma_+$. Let $H:=(2\beta+1)/c$ and split the sum into
\begin{align*}
\Sigma_+ = 
\sum_{\substack{a\in A_x\\ j\geq1\\ J-j<H}}d(n(a))^m +
\sum_{\substack{a\in A_x\\ j\geq1\\ J-j\geq H}}d(n(a))^m 
=:\Sigma_{+,<}+\Sigma_{+,\geq}.
\end{align*}
If $J-j<H$, then
\begin{align*}
d(n(a))^m \leq d(P_j)^m2^{m(J-j)} \ll d(P_j)^m
=\sum_{d_1,\ldots,d_m\mid P_j}1
\leq
\sum_{\substack{d_1,\ldots,d_m\leq x^c \\ d_1,\ldots,d_m\mid n(a)}}1.
\end{align*}
Consequently,
\begin{align*}
\Sigma_{+,<}
&\ll
\sum_{d_1,\ldots,d_m\leq x^c}
\pi\big(x,\lcm(d_1,\ldots,d_m)\big).
\end{align*}
For $\delta=\lcm(d_1,\ldots,d_m)$ in this sum,
\begin{align*}
\delta\leq d_1\cdots d_m\leq x^{mc}=x^{c_0}.
\end{align*}
Conditions (U2) and (U2-1) imply
\begin{align*}
\Sigma_{+,<} 
\ll \#A_x\sum_{\substack{d_1,\ldots,d_m\leq x^c\\ \delta:=\lcm(d_1,\ldots,d_m)}} \frac{1}{\psi(\delta)} 
\ll \#A_x(\log x)^{s(2^m-1)}.
\end{align*}

It remains to estimate $\Sigma_{+,\geq}$. Assume that $J-j\geq H$. Note that
\begin{align*}
p_{j+1}<x^{c/2}.
\end{align*}
Indeed, if $p_{j+1}\geq x^{c/2}$, then one has
\begin{align*}
n(a) \geq p_{j+1}\cdots p_J
\geq x^{c(J-j)/2}
\geq x^{\beta+1/2},
\end{align*}
contrary to (U1). By the definition of $j$,
\begin{align*}
P_jp_{j+1}>x^c,
\end{align*}
and hence
\begin{align*}
P_j>x^{c/2}.
\end{align*}

Choose the positive integer $u=u(a)$ determined by
\begin{align*}
x^{1/(u+1)}\leq p_j<x^{1/u}.
\end{align*}
Since $p_{j+1},\ldots,p_J\geq p_j$, we have
\begin{align*}
n(a)
\geq p_{j+1}\cdots p_J
\geq x^{(J-j)/(u+1)}.
\end{align*}
Condition (U1) now gives
\begin{align*}
J-j\leq\beta(u+1).
\end{align*}
Therefore,
\begin{align}\label{eqn_generalized tail divisor bound}
d(n(a))^m \leq 2^{m\beta(u+1)}d(P_j)^m.
\end{align}

The elementary symmetry of the divisors of $P_j$ implies
\begin{align*}
d(P_j) \leq 2\#\{d\mid P_j: d\geq\sqrt{P_j}\}.
\end{align*}
Since $P_j>x^{c/2}$, every divisor counted on the right is larger than $x^{c/4}$. Moreover, every prime factor of $P_j$ is at most $p_j<x^{1/u}$. It follows that
\begin{align}\label{eqn_generalized head divisor bound}
d(P_j)^m
\ll \sum_{\substack{d_1,\ldots,d_m\in S(x^c,x^{1/u})\\ d_1,\ldots,d_m\mid P_j\\ d_1,\ldots,d_m\geq x^{c/4}}}1 
\leq \sum_{\substack{d_1,\ldots,d_m\in S(x^c,x^{1/u})\\ \lcm(d_1,\ldots,d_m)\mid n(a)\\ d_1,\ldots,d_m\geq x^{c/4}}}1.
\end{align}
We divide $\Sigma_{+,\geq}$ according as
\begin{align*}
u<\frac{\log x}{2\log\log x}
\end{align*}
or not, and write
\begin{align*}
\Sigma_{+,\geq} = \Sigma_{+,\geq,1}+\Sigma_{+,\geq,2}.
\end{align*}

From \eqref{eqn_generalized tail divisor bound} and
\eqref{eqn_generalized head divisor bound}, condition (U2) yields
\begin{align*}
\Sigma_{+,\geq,1}
\ll \#A_x\sum_{u<\frac{\log x}{2\log\log x}} 2^{m\beta u}
\sum_{\substack{d_1,\ldots,d_m\in S(x^c,x^{1/u})\\ d_1,\ldots,d_m\geq x^{c/4}\\ \delta:=\lcm(d_1,\ldots,d_m)}} \frac{1}{\psi(\delta)}.
\end{align*}
There are only finitely many integers $u<4/c$. For those values, we have
\begin{align*}
\sum_{\substack{d_1,\ldots,d_m\in S(x^c,x^{1/u})\\ d_1,\ldots,d_m\geq x^{c/4}\\ \delta:=\lcm(d_1,\ldots,d_m)}} \frac{1}{\psi(\delta)}
\leq \sum_{\substack{d_1,\ldots,d_m\leq x^c\\ \delta:=\lcm(d_1,\ldots,d_m)}}\frac{1}{\psi(\delta)}
\ll (\log x)^{s(2^m-1)}.
\end{align*}
So we suppose
\begin{align*}
\frac{4}{c}\leq u<\frac{\log x}{2\log\log x}.
\end{align*}
For the corresponding inner sum, the Cauchy--Schwarz inequality gives
\begin{align*}
&\sum_{\substack{d_1,\ldots,d_m\in S(x^c,x^{1/u})\\ d_1,\ldots,d_m\geq x^{c/4}\\ \delta:=\lcm(d_1,\ldots,d_m)}} \frac{1}{\psi(\delta)} \\
&\leq \bigg(\sum_{\substack{d_1,\ldots,d_m\in S(x^c,x^{1/u})\\ d_1,\ldots,d_m\geq x^{c/4}\\ \delta:=\lcm(d_1,\ldots,d_m)}}\frac{\mathcal G(\delta)}{\delta}\bigg)^{1/2} 
\times
\bigg(\sum_{\substack{d_1,\ldots,d_m\in S(x^c,x^{1/u})\\ d_1,\ldots,d_m\geq x^{c/4}\\ \delta:=\lcm(d_1,\ldots,d_m)}} \frac{\delta}{\psi(\delta)^2\mathcal G(\delta)}\bigg)^{1/2}.
\end{align*}
The second factor can be enlarged by removing the restrictions on the $d_i$. Thus, (U2-2) and (U2-3) imply
\begin{align*}
\sum_{\substack{d_1,\ldots,d_m\in S(x^c,x^{1/u})\\ d_1,\ldots,d_m\geq x^{c/4}\\ \delta:=\lcm(d_1,\ldots,d_m)}} \frac{1}{\psi(\delta)}
\ll (\log x)^{s(2^m-1)}\exp\left(-\frac{cu\log u}{32}\right).
\end{align*}
We conclude that
\begin{align*}
\Sigma_{+,\geq,1} 
\ll \#A_x(\log x)^{s(2^m-1)} \left(1+ \sum_{u\geq4/c} \exp\left(m\beta u\log2-\frac{cu\log u}{32}\right)\right) \ll \#A_x(\log x)^{s(2^m-1)},
\end{align*}
since the series in the parentheses is convergent.

We finally consider $\Sigma_{+,\geq,2}$. In this range,
\begin{align*}
u\geq\frac{\log x}{2\log\log x}.
\end{align*}
A well-known estimate for the number of divisors
\begin{align*}
    d(n) \ll \exp\left(\frac{\log n}{\log\log n}\right)
\end{align*}
implies that there is a constant $C_{m,\beta}>0$,
\begin{align*}
d(p_{j+1}\cdots p_J)^m
\ll \exp\left(\frac{C_{m,\beta}\log x}{\log\log x}\right).
\end{align*}
Combining this with the estimate for the divisors of $P_j$, we obtain
\begin{align*}
d(n(a))^m
\ll \exp\left(\frac{C_{m,\beta}\log x}{\log\log x}\right)
\sum_{\substack{d_1,\ldots,d_m\in S(x^c,x^{1/u})\\ \lcm(d_1,\ldots,d_m)\mid n(a)\\ d_1,\ldots,d_m\geq x^{c/4}}} 1.
\end{align*}
Since $p_j\geq2$ and $p_j<x^{1/u}$, necessarily
\begin{align*}
u\leq\frac{\log x}{\log2}.
\end{align*}
After summing over $a$ and applying (U2), it follows that
\begin{align*}
\Sigma_{+,\geq,2}
\ll \#A_x\exp\left(\frac{C_{m,\beta}\log x}{\log\log x}\right)
\sum_{\frac{\log x}{2\log\log x}\leq u\leq\frac{\log x}{\log2}}
\sum_{\substack{d_1,\ldots,d_m\in S(x^c,x^{1/u})\\ d_1,\ldots,d_m\geq x^{c/4}\\ \delta:=\lcm(d_1,\ldots,d_m)}}\frac{1}{\psi(\delta)}.
\end{align*}
For sufficiently large $x$, every $\delta$ occurring here satisfies
\begin{align*}
\delta\geq x^{c/4}\geq3.
\end{align*}
By (U2-4),
\begin{align*}
\frac{1}{\psi(\delta)}
=\frac{\mathcal G(\delta)}{\delta} \frac{\delta}{\psi(\delta)\mathcal G(\delta)}
\ll (\log\log\delta)^B\frac{\mathcal G(\delta)}{\delta}
\ll (\log\log x)^B\frac{\mathcal G(\delta)}{\delta}.
\end{align*}
Condition (U2-3) implies
\begin{align*}
\Sigma_{+,\geq,2} \ll
\#A_x x^{-\eta}(\log x)(\log\log x)^B \exp\left(\frac{C_{m,\beta}\log x}{\log\log x}\right) \ll \#A_x.
\end{align*}

Collecting the estimates for $\Sigma_0$, $\Sigma_{+,<}$,
$\Sigma_{+,\geq,1}$ and $\Sigma_{+,\geq,2}$, we obtain
\begin{align*}
\sum_{a\in A_x}d(n(a))^m
\ll \#A_x(\log x)^{s(2^m-1)}.
\end{align*}
This completes the proof of the upper bound.

\subsection{Lower bound (Proof of Theorem~\ref{thm_auxiliary lower})}

Let $c':=c_1/m$. Using the divisor expansion and retaining only the tuples whose least common multiple belongs to $\mathbb N_q$, we obtain
\begin{align*}
\sum_{a\in A_x}d(n(a))^m &=\sum_{a\in A_x}\sum_{d_1,\ldots,d_m\mid n(a)}1 \geq \sum_{a\in A_x} \sum_{\substack{d_1,\ldots,d_m\leq x^{c'}\\ d_1,\ldots,d_m\mid n(a)\\ \lcm(d_1,\ldots,d_m)\in\mathbb N_q}}1 = \sum_{\substack{d_1,\ldots,d_m\leq x^{c'}\\ \delta:=\lcm(d_1,\ldots,d_m)\in\mathbb N_q}} \pi(x,\delta).
\end{align*}
For every tuple in the last sum,
\begin{align*}
\delta\leq d_1\cdots d_m\leq x^{mc'}=x^{c_1}.
\end{align*}
Consequently, condition (L) is applicable, and hence for all sufficiently large $x$,
\begin{align*}
\sum_{a\in A_x}d(n(a))^m &\gg \#A_x \sum_{\substack{d_1,\ldots,d_m\leq x^{c'}\\
\delta:=\lcm(d_1,\ldots,d_m)\in\mathbb N_q}} \frac{1}{\varrho(\delta)}.
\end{align*}
Applying (L-1) with $x^{c'}$ in place of $x$ gives
\begin{align*}
\sum_{\substack{d_1,\ldots,d_m\leq x^{c'}\\
\delta:=\lcm(d_1,\ldots,d_m)\in\mathbb N_q}}
\frac{1}{\varrho(\delta)}
\gg
\big(\log x^{c'}\big)^{s(2^m-1)}
\gg
(\log x)^{s(2^m-1)}.
\end{align*}
It follows that
\begin{align*}
\sum_{a\in A_x}d(n(a))^m
\gg
\#A_x(\log x)^{s(2^m-1)}.
\end{align*}
This proves the lower bound.

\section{Mean value of multivariable multiplicative function with prime class-dependent local data}\label{sec_tauberian}

In this section, we generalize the results in \cite{EST22} to the situation with finitely many prime class-dependent local data. In other words, we will prove Theorem~\ref{thm_multivariable tauberian with many local data} below.

We introduce the following notations:
\begin{itemize}
    \item For $\nu=(\nu_1,\ldots,\nu_m) \in \mathbb N_0^m$, $|\nu|_1:=|\nu_1|+\cdots+|\nu_m|$, $|\nu|_{\infty}:=\max_{1\leq i \leq m} |\nu_i|$.
    \item A function $\varphi: \mathbb N^m \to \mathbb C$ is said to be (multivariable) multiplicative if for all $\mathbf d=(d_1,\ldots,d_m) \in \mathbb N^m$ and $\mathbf d'=(d_1',\ldots,d_m') \in \mathbb N^m$ satisfying $\mathrm{gcd}(\mathrm{lcm}(d_1,\ldots,d_m),\mathrm{lcm}(d_1',\ldots,d_m'))=1$, we have $\varphi(d_1d_1',\ldots,d_md_m')=\varphi(\mathbf d)\varphi(\mathbf d')$.
\end{itemize}

\begin{definition}\label{def_quintuple data_irred ver}
    A quintuple $(g,\kappa,\mathbf c,\xi, P'')$ is said to be data if the following hold:
    \begin{enumerate}
        \item[\normalfont(i)] A function $g \colon \mathbb N_0^m \to \mathbb R_{\geq 0}$ that has subexponential growth, i.e., for any $\epsilon>0$, we have $g(\nu) \ll_{\epsilon} e^{\epsilon|\nu|_1}$ uniformly in $\nu \in \mathbb N_0^m$.
        \item[\normalfont(ii)] A function $\kappa \colon \mathbb N_0^m \to [1,\infty) \cup \{0\}$ that satisfies $\kappa(0)=0$ and $\mathrm{inf}_{\nu \in \mathbb N_0^m \setminus \{0\}} \frac{\kappa(\nu)}{|\nu|_1}>0$.
        \item[\normalfont(iii)] An $m$-tuple $c=(c_1,\ldots,c_m) \in [0,\infty)^m$.
        \item[\normalfont(iv)] A real number $\xi \in  (0,\infty)$.
        \item[\normalfont(v)] A set
        \begin{align*}
            I(\kappa,g):=\{\nu \in \mathbb N_0^m: \ \kappa(\nu)=1 \ \text{and} \ g(\nu) \neq 0\}
        \end{align*}
        is finite and non-empty.
        \item[\normalfont(vi)] A subset $P''$ of $\mathbb P$.
        \item[\normalfont(vii)] There exists a positive real number $D(P'')>0$ such that for the partial zeta function associated to $P''$
        \begin{align*}
            \zeta_{P''}(w)=\prod_{p\in P''}(1-p^{-w})^{-1},
        \end{align*}
        the function $w\zeta_{P''}(1+w)^{D(P'')}$ continues to a nonzero holomorphic function in the region defined by $\Re(w)>-\epsilon_0$ for some $\epsilon_0>0$.
        \item[\normalfont(viii)] There exists a positive real number $C(P'')>0$ such that $\zeta_{P''}(w)$ satisfies the convexity bound
        \begin{align*}
            w\zeta_{P''}(1+w)^{D(P'')} \ll_{\epsilon} (1+|w|)^{1-\frac{C(P'')D(P'')}{2}\min\{0,\Re(w)\}+\epsilon}
        \end{align*}
        in the region defined by $\Re(w)\geq -\epsilon_0$ for some $\epsilon_0>0$.
    \end{enumerate}
\end{definition}

\begin{remark}
    If we choose $P''=\mathbb P$, then the conditions (vii) and (viii) in Definition~\ref{def_quintuple data_irred ver} are automatically fulfilled with $C(\mathbb P)=D(\mathbb P)=1$. Thus, a quintuple data $(g,\kappa,\mathbf c,\xi,\mathbb P)$ can be seen as a quadruple data $(g,\kappa,\mathbf c,\xi)$ in the sense of \cite{EST22}.  Note that one can replace $\mathbb P$ with $\mathbb P_{>q}$ for any fixed prime $q$.
\end{remark}

\begin{definition}
    Let $(g,\kappa,\mathbf c, \xi, P'')$ be data defined in Definition~\ref{def_quintuple data_irred ver}. A multiplicative function $\varphi: \mathbb N^m \to \mathbb C$ is said to be in the class $\mathcal C(g,\kappa,\mathbf c,\xi,P'')$ if the following conditions hold:
    \begin{enumerate}
        \item[\normalfont(i)] For any $\epsilon>0$, 
            \begin{align*}
                \varphi(p^{\nu_1},\ldots,p^{\nu_m})-g(\nu)p^{\langle c,\nu \rangle-\kappa(\nu)} \ll_{\epsilon} e^{\epsilon |\nu|_1}p^{\langle c,\nu \rangle-\kappa(\nu)-\xi}
            \end{align*}
        uniformly in $\nu \in \mathbb N_0^m$ and $p \in P''$.
        \item[\normalfont(ii)] If $p \not\in P''$, then 
        \begin{align*}
            \varphi(p^{\nu_1},\ldots,p^{\nu_m})=0
        \end{align*}
        for any $\nu=(\nu_1,\ldots,\nu_m) \in \mathbb N_0^m$.
    \end{enumerate}
\end{definition}

For convenience, we record de la Bret\`eche's multivariable Tauberian theorem that will be used below. The formulation is adapted from \cite[\S~4.1, Theorems~A and~B]{EST22}, which reformulates \cite[Theorems~1 and~2]{Bre01}, and the notation has been adjusted to the present setting.

Let
\[
    \mathcal{LR}_{m}^{+}(\mathbb C)
    :=
    \left\{
        \ell\in\operatorname{Hom}_{\mathbb C}
        (\mathbb C^m,\mathbb C):
        \ell\bigl([0,\infty)^m\bigr)\subseteq[0,\infty)
    \right\}.
\]

\begin{theorem}{\cite[Theorem~1]{Bre01}}\label{thm_Bre A}
Let
\[
    \varphi:\mathbb N^m\longrightarrow[0,\infty)
\]
be a nonnegative function, and let
\[
    \mathcal F(\mathbf w)
    =
    \mathcal F(w_1,\ldots,w_m)
    :=
    \sum_{d_1,\ldots,d_m\geq1}
    \frac{\varphi(d_1,\ldots,d_m)}
         {d_1^{w_1}\cdots d_m^{w_m}}
\]
be its associated multiple Dirichlet series.

Suppose that there exists
\[
    \mathbf c=(c_1,\ldots,c_m)\in[0,\infty)^m
\]
such that the following conditions hold.

\begin{enumerate}
    \item[\normalfont(i)]
    The series $\mathcal F(\mathbf w)$ converges absolutely whenever
    \[
        \Re(w_j)>c_j
        \qquad (1\leq j\leq m).
    \]

    \item[\normalfont(ii)]
    There exist finite families
    \[
        \mathcal L
        =
        \bigl(\ell^{(i)}\bigr)_{1\leq i\leq q}
        \subseteq
        \mathcal{LR}_{m}^{+}(\mathbb C)\setminus\{0\},
    \]
    and
    \[
        \bigl(h^{(i)}\bigr)_{1\leq i\leq q'}
        \subseteq
        \mathcal{LR}_{m}^{+}(\mathbb C),
    \]
    together with constants $\delta_1,\delta_2,\delta_3>0$, such that
    \[
        \mathcal H(\mathbf w)
        :=
        \mathcal F(\mathbf c+\mathbf w)
        \prod_{i=1}^{q}\ell^{(i)}(\mathbf w)
    \]
    admits a holomorphic continuation to
    \[
        \mathcal D(\delta_1,\delta_3)
        :=
        \left\{
            \mathbf w\in\mathbb C^m:
            \begin{array}{ll}
            \Re\!\left(\ell^{(i)}(\mathbf w)\right)>-\delta_1
                & (1\leq i\leq q),\\[2mm]
            \Re\!\left(h^{(i)}(\mathbf w)\right)>-\delta_3
                & (1\leq i\leq q')
            \end{array}
        \right\}.
    \]
    Moreover, for every $\varepsilon,\varepsilon'>0$, uniformly for
    \[
        \mathbf w\in
        \mathcal D(\delta_1-\varepsilon',
                   \delta_3-\varepsilon') \text{ with } \Re(w_j) \ll 1, \ 1\leq j \leq m 
    \]
    one has
    \begin{align}\label{eqn_H(w) bound in Bre A}
        \mathcal H(\mathbf w)
        \ll_{\varepsilon,\varepsilon'}
        &\prod_{i=1}^{q}
        \left(
            1+
            \left|
                \Im\!\left(\ell^{(i)}(\mathbf w)\right)
            \right|
        \right)^{
            1-\delta_2
            \min\left\{
                0,
                \Re\!\left(\ell^{(i)}(\mathbf w)\right)
            \right\}
        }                                      \\
        &\times
        \left(
            1+
            \left(
                |\Im(w_1)|+\cdots+|\Im(w_m)|
            \right)^{\varepsilon}
        \right).
    \end{align}
\end{enumerate}

Set
\[
    J=J(\mathbf c)
    :=
    \left\{
        j\in\{1,\ldots,m\}:c_j=0
    \right\}.
\]
Write
\[
    j_1<\cdots<j_{\#J}
\]
for the elements of $J$, and augment the family $\mathcal L$ by the
$\#J$ linear forms
\[
    \ell^{(q+i)}(\mathbf w)
    :=
    \mathbf e_{j_i}^{*}(\mathbf w)
    =
    w_{j_i},
    \qquad 1\leq i\leq\#J,
\]
where
$(\mathbf e_1^{*},\ldots,\mathbf e_m^{*})$
is the dual basis of $(\mathbb C^m)^{*}$.

Then, for every
\[
    \boldsymbol{\beta}
    =
    (\beta_1,\ldots,\beta_m)
    \in(0,\infty)^m,
\]
there exist a polynomial
$Q_{\boldsymbol{\beta}}\in\mathbb R[X]$
and a constant $\theta>0$ such that
\begin{align*}
    \sum_{1\leq d_1\leq x^{\beta_1}}
    \cdots
    \sum_{1\leq d_m\leq x^{\beta_m}}
    f(d_1,\ldots,d_m)
    =
    x^{\langle\mathbf c,\boldsymbol{\beta}\rangle}
    Q_{\boldsymbol{\beta}}(\log x) +
    O\!\left(
        x^{\langle\mathbf c,\boldsymbol{\beta}\rangle-\theta}
    \right)
\end{align*}
as $x\to\infty$, and
\[
    \deg Q_{\boldsymbol{\beta}}
    \leq
    q+\#J
    -
    \operatorname{Rank}
    \left(
        \ell^{(1)},\ldots,\ell^{(q+\#J)}
    \right).
\]
\end{theorem}

\begin{remark}
In contrast to the growth condition \eqref{eqn_H(w) bound in Bre A} of Theorem~\ref{thm_Bre A}, the original statement \cite[Theorem~1]{Bre01} imposes the corresponding condition on the whole domain $\mathcal D(\delta_1-\varepsilon',\delta_3-\varepsilon')$. This is stronger than what is actually used in the proof. Indeed, in the contour shifting argument in \cite{Bre01}, the real parts of the integration variables range only over fixed bounded intervals, so the growth estimate is used only on vertical strips. The same observation applies throughout the inductive residue argument. It is worth noting that this distinction is also consistent with the examples following \cite[Theorem~1]{Bre01}. For example, the literal global growth condition is not satisfied by $w\zeta(1+w)$ along the positive real axis; nevertheless, this function is among the basic examples to which the theorem is applied.
\end{remark}

\begin{theorem}{\cite[Theorem~2(ii),(iv)]{Bre01}}\label{thm_Bre B}
Assume the hypotheses of Theorem~\ref{thm_Bre A}.  For $\boldsymbol{\beta}=(\beta_1,\ldots,\beta_m)\in(0,\infty)^m$, define the linear form
\[
    \mathcal B
    :=
    \sum_{j=1}^{m}\beta_j\mathbf e_j^{*}
    \in\mathcal{LR}_{m}^{+}(\mathbb C),
\]
and put
\[
    \mathcal L_0
    :=
    \bigl(
        \ell^{(1)},\ldots,\ell^{(q+\#J)}
    \bigr),
    \qquad
    \rho
    :=
    q+\#J-\operatorname{Rank}(\mathcal L_0).
\]
We use the notation
\[
    \operatorname{con}^{*}(\mathcal L_0)
    :=
    \left\{
        \sum_{k=1}^{q+\#J}
        \lambda_k\ell^{(k)}:
        \lambda_k>0
        \ \text{for every }k
    \right\}.
\]

\begin{enumerate}
    \item[\normalfont(a)]
    Suppose that the following conditions are satisfied.

    \begin{enumerate}
        \item[\normalfont(C1)]
        There exists a function $G$ such that
        \[
            \mathcal H(\mathbf w)
            =
            G\!\left(
                \ell^{(1)}(\mathbf w),
                \ldots,
                \ell^{(q+\#J)}(\mathbf w)
            \right).
        \]

        \item[\normalfont(C2)]
        One has
        \[
            \mathcal B
            \in
            \operatorname{span}_{\mathbb C}(\mathcal L_0),
        \]
        and there is no proper subfamily
        $\mathcal L'\subsetneq\mathcal L_0$ such that
        \[
            \mathcal B
            \in
            \operatorname{span}_{\mathbb C}(\mathcal L')
        \]
        and
        \[
            \#\mathcal L'
            -
            \operatorname{Rank}(\mathcal L')
            =
            \#\mathcal L_0
            -
            \operatorname{Rank}(\mathcal L_0).
        \]
    \end{enumerate}

    Then the polynomial furnished by
    Theorem~\ref{thm_Bre A} satisfies
    \[
        Q_{\boldsymbol{\beta}}(\log x)
        =
        \mathcal H(\mathbf 0)\,
        x^{-\langle\mathbf c,\boldsymbol{\beta}\rangle}
        \mathcal I_{\boldsymbol{\beta}}(x)
        +
        O\!\left((\log x)^{\rho-1}\right),
    \]
    where
    \[
        \mathcal I_{\boldsymbol{\beta}}(x)
        :=
        \int_{\mathcal A_{\boldsymbol{\beta}}(x)}
        \frac{dy_1\cdots dy_q}
        {\displaystyle
            \prod_{i=1}^{q}
            y_i^{\,1-\ell^{(i)}(\mathbf c)}
        },
    \]
    and
    \[
        \mathcal A_{\boldsymbol{\beta}}(x)
        :=
        \left\{
            \mathbf y=(y_1,\ldots,y_q)\in[1,\infty)^q:
            \prod_{i=1}^{q}
            y_i^{\,\ell^{(i)}(\mathbf e_j)}
            \leq x^{\beta_j}
            \quad (1\leq j\leq m)
        \right\}.
    \]

    \item[\normalfont(b)]
    Suppose that
    \[
        \operatorname{Rank}(\mathcal L_0)=m,
        \qquad
        \mathcal H(\mathbf 0)\neq0,
        \qquad
        \mathcal B\in\operatorname{con}^{*}(\mathcal L_0).
    \]
    Then
    \[
        \deg Q_{\boldsymbol{\beta}}
        =
        \rho
        =
        q+\#J-m.
    \]
\end{enumerate}

The assumptions in part {\rm(b)} imply those in part {\rm(a)}.
\end{theorem}

It is important to note that de la Bret\`eche's multivariable Tauberian theorem does not compute the mean value of a multivariable multiplicative function in the class $\mathcal C(g,\kappa,\c,\xi,P'')$. The main reason is that $\lambda(\nu)$ is not an integer. In order to apply Theorem~\ref{thm_Bre A} and \ref{thm_Bre B}, the integral pole order ($\lambda(\nu)$ in our case) is required. But this assumption is too strong for our purpose.

Rather than compute the mean value with restricted prime support, we compute the mean value with finitely many prime class-dependent local data as follows. 

For a multivariable multiplicative function $\varphi:\N^m \to \C$, consider the Euler product of its Dirichlet series
\begin{align*}
    \mathcal M(\varphi;\mathbf w)=\sum_{d_1,\ldots,d_m \geq 1}\frac{\varphi(d_1,\ldots,d_m)}{d_1^{w_1}\cdots d_m^{w_m}}=\prod_p\sum_{\nu \in \N_0^m}\frac{\varphi(p^{\nu_1},\ldots,p^{\nu_m})}{p^{\langle \nu,\w \rangle}}.
\end{align*}
If $\mathbb P=\mathcal Q_1 \sqcup \cdots \sqcup \mathcal Q_t$, then we have
\begin{align*}
    \prod_p\sum_{\nu \in \N_0^m}\frac{\varphi(p^{\nu_1},\ldots,p^{\nu_m})}{p^{\langle \nu,\w \rangle}}= \prod_{p \in \mathcal Q_1}\sum_{\nu \in \N_0^m}\frac{\varphi(p^{\nu_1},\ldots,p^{\nu_m})}{p^{\langle \nu,\w \rangle}} \times \cdots \times \prod_{p \in \mathcal Q_t}\sum_{\nu \in \N_0^m}\frac{\varphi(p^{\nu_1},\ldots,p^{\nu_m})}{p^{\langle \nu,\w \rangle}}.
\end{align*}
For each $i$, the partial Euler product
\begin{align*}
    \prod_{p \in \mathcal Q_i}\sum_{\nu \in \N_0^m}\frac{\varphi(p^{\nu_1},\ldots,p^{\nu_m})}{p^{\langle \nu,\w \rangle}}=\sum_{d_1,\ldots,d_m\geq 1} \frac{\varphi_{\mathcal Q_i}(d_1,\ldots,d_m)}{d_1^{w_1}\cdots d_m^{w_m}}
\end{align*}
defines a multivariable multiplicative function $\varphi_{\mathcal Q_i}$. More explicitly, this $\varphi_{\mathcal Q_i}$ is defined by
\begin{align}
        \label{eqn_varphi_Q}\varphi_{\mathcal Q_i}(d_1,\ldots,d_m):=\begin{cases}
            \varphi(d_1,\ldots,d_m) & \text{if} \ d_1,\ldots,d_m \in \mathcal J_i  \ \text{and} \ (d_1,\ldots,d_m) \neq (1,\ldots,1), \\
            1 & \text{if} \ (d_1,\ldots,d_m)=(1,\ldots,1), \\
            0 & \text{Otherwise,}
            \end{cases}
    \end{align}
where $\mathcal J_i:=\{\delta \in \N: \text{every prime factor of $\delta$ belongs to $\mathcal Q_i$}\}$.

\begin{theorem}\label{thm_multivariable tauberian with many local data}
    Let $\varphi: \N^m \to \R_{\geq 0}$ be a nonnegative multivariable multiplicative function. Suppose that $\mathcal Q_1,\ldots,\mathcal Q_t$ form a partition of $\mathbb P$ such that $(g_{i},\kappa,\c,\xi_i,\mathcal Q_i)$ is a quintuple data for $1\leq i\leq t$ with
    \[
        \Lambda(\nu):=\sum_{i=1}^t \lambda_i(\nu)=\sum_{i=1}^t \frac{g_i(\nu)}{D(\mathcal Q_i)} \in \mathbb N_0.
    \]
    For each $1\leq i\leq t$, let
    \[
    I_i:=I(\kappa,g_i),
    \]
    and put
    \[
    I:=\bigcup_{i=1}^t I_i=\{\nu\in\mathbb N_0^m: \kappa(\nu)=1,\ \Lambda(\nu)>0\}.
    \]
    Let
    \[
    J:=\{\mathbf e_j\in\mathbb R^m:c_j=0\}.
    \]
    Assume that
    \begin{enumerate}
        \item[\normalfont(i)]
        $\operatorname{Rank}(I\cup J)=m$;
        \item[\normalfont(ii)]
        $\mathbf 1=(1,\ldots,1)$ belongs to
        $\operatorname{con}^{*}(I\cup J)$.
    \end{enumerate}
    If $\varphi_{\mathcal Q_i} \in \mathcal C(g_{i},\kappa,\c,\xi_i;\mathcal Q_i)$, then 
    \begin{align}\label{eqn_whole f(d_1,...,d_m)}
        \sum_{d_1,\ldots,d_m \leq x} \varphi(d_1,\dots,d_m)=C(\varphi)x^{|\c|_1}(\log x)^{\rho}+O(x^{|\c|_1}(\log x)^{\rho-1})
    \end{align}
    for some constant $C(\varphi)>0$ and
    \[
    \rho=\sum_{\nu \in I} \Lambda(\nu) -\#\{j \in \{1,2,\ldots,m\}:c_{j}\neq 0\}.
    \]
\end{theorem}

\begin{proof}
    Define
    \[
        \mathcal M_i(\w):=\prod_{p \in \mathcal Q_i}\sum_{\mu \in \N_0^m}\frac{\varphi(p^{\mu_1},\ldots,p^{\mu_m})}{p^{\langle \mu,\c+\w \rangle}}
    \]
    so that
    \[
        \mathcal M(\varphi;\c+\w)=\prod_p\sum_{\mu \in \N_0^m}\frac{\varphi(p^{\mu_1},\ldots,p^{\mu_m})}{p^{\langle \mu,\w \rangle}}=\prod_{i=1}^t \mathcal M_i(\w).
    \]
    For each $1\leq i \leq t$, let $I_i:=I(\kappa,g_i)$ and
    \[
        \mathcal E_i(\w):=\left(\prod_{\nu \in I_i}\zeta_{\mathcal Q_i}(1+\langle \nu,\mathbf w\rangle)^{-g_i(\nu)}\right) \mathcal M_i(\mathbf w).
    \]
    We have the formal Euler product
    \begin{align*}
        \mathcal E_i(\w)&=\prod_{p \in \mathcal Q_i}\mathcal E_{i,p}(\w), \\
        \mathcal E_{i,p}(\w)&:=\prod_{\nu \in I_i}\left(1-\frac{1}{p^{1+\langle\nu,\mathbf w \rangle}}\right)^{g_i(\nu)}\times \sum_{\mu \in \mathbb N_0^m}\frac{\varphi(p^{\mu_1},\ldots,p^{\mu_m})}{p^{\langle\mu,\mathbf c+\mathbf w\rangle}}.
    \end{align*}
    By applying \cite[Lemma~1,2]{EST22}, we obtain the following analogue of \cite[Lemma~3]{EST22}; there exists $\epsilon_0>0$ such that $\mathcal E_i(\w)$ converges absolutely and extends to a bounded holomorphic function in the tube domain $U_{-\epsilon_0}:=\{\mathbf w \in \mathbb C^m: \ \Re(w_i)>-\epsilon_0 \text{ for any } i=1,\ldots,m\}$. 

    Let
    \[
        \mathcal E(\w):=\prod_{i=1}^t \mathcal E_i(\w),
    \]
    which is a bounded holomorphic function in $U_{-\epsilon_0}$. We also let
    \begin{align*}
        A_i(w)&:=w\zeta_{\mathcal Q_i}(1+w)^{D(\mathcal Q_i)}, \\
        B_{\nu}(w)&:=\prod_{i=1}^t A_i(w)^{\lambda_i(\nu)}=w^{\Lambda(\nu)}\prod_{i=1}^t \zeta_{\mathcal Q_i}(1+w)^{g_i(\nu)}.
    \end{align*}
    As $(g_i,\kappa,\c,\xi_i,\mathcal Q_i)$ is a quintuple data, $B_{\nu}(w)$ is a nonzero holomorphic function in $U_{\epsilon_0}$ that satisfies the convexity bound
    \begin{align*}
        B_{\nu}(w) \ll_{\epsilon} (1+|w|)^{\Lambda(\nu)-\frac{1}{2}(\sum_{i=1}^t C(\mathcal Q_i)g_i(\nu))\min\{0,\Re(w)\}+\epsilon}
    \end{align*}

    We define the function
    \[
        H(\varphi,\c;\w):=\left(\prod_{\nu \in I}\langle \nu,\mathbf w\rangle^{ \Lambda(\nu)}\right)\mathcal M(\varphi;\mathbf c+\mathbf w).
    \]
    Then
    \begin{align*}
        H(\varphi,\c;\w)&=\prod_{i=1}^t \left(\prod_{\nu \in I_i} \langle \nu,\mathbf w\rangle^{ \lambda_i(\nu)}\right)\mathcal M_i(\w) \\
        &=\prod_{i=1}^t \left(\prod_{\nu \in I_i}\zeta_{\mathcal Q_i}(1+\langle \nu,\mathbf w\rangle)^{-g_i(\nu)} \right)\left(\prod_{\nu \in I_i}\langle \nu,\mathbf w\rangle)^{\lambda_i(\nu)}\zeta_{\mathcal Q_i}(1+\langle \nu,\mathbf w\rangle)^{g_i(\nu)} \right)\mathcal M_i(\w) \\
        &=\prod_{i=1}^t \left(\prod_{\nu \in I_i}\zeta_{\mathcal Q_i}(1+\langle \nu,\mathbf w\rangle)^{-g_i(\nu)} \right)\left(\prod_{\nu \in I_i}A_i(\langle \nu,\mathbf w\rangle)^{\lambda_i(\nu)} \right)\mathcal M_i(\w) \\
        &=\prod_{i=1}^t \mathcal E_i(\w)\prod_{\nu \in I_i}A_i(\langle \nu,\mathbf w\rangle)^{\lambda_i(\nu)} =\mathcal E(\w)\prod_{\nu \in I}B_{\nu}(\langle\nu,\w\rangle).
    \end{align*}
    
    We apply Theorem~\ref{thm_Bre A} and \ref{thm_Bre B} with $\mathcal F(\w)=\mathcal M(\varphi;\w)$ and $\mathcal H(\w)=\mathcal H(\varphi,\c;\w)$. Let $I=\{\nu^1,\ldots,\nu^r\}$,
    \[
        q_0:=0, \qquad q_k=\sum_{j=1}^k \Lambda(\nu^j)\in \N_0 \quad \text{ for } k=1,\ldots,r, \qquad q:=q_r
    \]
    and
    \begin{align*}
        &\ell^{(i)}(\w):=
            \langle \nu^k,\w \rangle & \quad &\text{for} \ q_{k-1}<i\leq q_k \ \text{and} \ 1\leq k\leq r, \\
        &\ell^{(q+i)}(\w):=\e_{j_i}^*(\w)=w_{j_i} & \quad &\text{for} \ 1\leq i\leq \#J, \\
        &h^{(i)}(\w):=\e_i^*(\w)=w_i & \quad &\text{for} \ 1\leq i\leq m,
    \end{align*}
    where $J=\{j_1,\ldots,j_{\#J}\}$ with $j_1<\ldots<j_{\#J}$. If we take $$\delta_1=\max_{1 \leq k\leq r} (\nu_1^k+\cdots+\nu_m^k)\epsilon_0, \quad \delta_2=\frac{1}{2}\sum_{i=1}^t C(\mathcal Q_i)g_i(\nu), \quad \delta_3=\epsilon_0,$$
    then 
    \begin{align*}
        \prod_{i=1}^q(1+|\ell^{(i)}(\w)|)^{\left(1-\frac{1}{2}(\sum_{i=1}^t C(\mathcal Q_i)g_i(\nu))\min\{0,\Re(\langle \mathbf \nu, \mathbf w\rangle)\}\right)} \ll \prod_{i=1}^q(1+|\Im(\ell^{(i)}(\w))|)^{\left(1-\delta_2\min\{0,\Re(\langle \mathbf \nu, \mathbf w\rangle)\}\right)}
    \end{align*}
    and
    \begin{align*}
        \prod_{i=1}^q(1+|\ell^{(i)}(\w)|)^{\epsilon} \ll \prod_{i=1}^q(1+|\Im(\ell^{(i)}(\w))|)^{\epsilon} \ll \left(1+|\Im(w_1)|+\cdots+|\Im(w_m)|\right)^{\epsilon''}.
    \end{align*}
    Thus the condition~(i) and (ii) of Theorem~\ref{thm_Bre A} are verified. 

    The three assumptions in Theorem~\ref{thm_Bre B}(b) are also satisfied; by duality, $\operatorname{Rank}(\mathcal L_0)=\operatorname{Rank}(I\cup J)=m$ and $\mathcal B=\e_1^*+\cdots+\e_m^* \in \operatorname{con}^*(I\cup J)$. If $\varphi(\unb)=0$, then we have $\varphi(\d)=0$ for any $\d$ by multiplicativity, so $\varphi(\unb)\neq 0$. Applying Theorem~\ref{thm_Bre A} and \ref{thm_Bre B} with $\boldsymbol{\beta}=(1,\ldots,1)$, we complete the proof.
\end{proof}


Standard examples of partitions $\mathcal Q_1,\ldots,\mathcal Q_t$ of $\mathbb P$ that fulfill the conditions in Theorem~\ref{thm_multivariable tauberian with many local data} arise from Frobenian sets. We say that $\mathcal Q \subset \mathbb P$ is Frobenian of density $\alpha>0$ if there exists a finite Galois extension $K/\mathbb Q$ and a subset $H$ of $G=\mathrm{Gal}(K/\mathbb Q)$ such that the following properties hold:
    \begin{enumerate}
        \item[\normalfont(i)] $H$ is stable under $G$-conjugation,
        \item[\normalfont(ii)] $|H|/|G|=\alpha$,
        \item[\normalfont(iii)] For $p\gg 1$, the Frobenius element $\mathrm{Frob}_p(K/\mathbb Q)$ belongs to $H$ if and only if $p \in \mathcal Q$.
    \end{enumerate}

\begin{lemma}\cite{Ser75}\label{lem_Ser Frobenian}
    Suppose that $\mathcal Q$ is Frobenian of density $\alpha>0$. 
    Then a function
    \begin{align*}
        w\zeta_{\mathcal Q}(1+w)^{1/\alpha}
    \end{align*}
    is continued to a nonzero holomorphic function in the tube domain defined by $\Re(w)\geq 0$.
\end{lemma}
    
Suppose that for $1\leq i\leq t$, $S_i$ is a subset of $G=\mathrm{Gal}(K/\Q)$ that is stable under $G$-conjugation, and let $\mathcal Q_i$ be $S_i$-Frobenian of density $\delta_i$. Let $D(\mathcal Q_i):=1/\delta_i$, and assume the conditions (vii) and (viii) of Definition~\ref{def_quintuple data_irred ver} that are stronger than the conclusion of Lemma~\ref{lem_Ser Frobenian}. For each $g_i$ and $\nu$, define a class function $\mathcal G_{i,\nu}:G \to \R_{\geq 0}$ by
$\mathcal G_{i,\nu}(\sigma)=g_i(\nu)\mathbf 1_{S_i}(\sigma).$ Suppose that $G$ acts on sets $X_1,\ldots, X_s$. If $\mathcal G_{i,\nu}$ is given by an $\N_0$-linear combination of the permutation characters\footnote{If a finite group $G$ acts on $X$, under the embedding $G$ into the symmetric group $S_{X}$, the number of fixed points $|{\rm Fix}_{X}|$ can be seen as the character of a permutation representation of $S_{X}$ (more precisely, its restriction to $G$). For this reason, we simply refer to a function of the form $|{\rm Fix}_{X}|$ as a permutation character.}, i.e.,
\[
    \mathcal G_{i,\nu}(\sigma)=n_1|{\rm Fix}_{X_1}(\sigma)|+\dots + n_s |{\rm Fix}_{X_s}(\sigma)|, \quad n_1,\ldots, n_s \in \N_0,
\]
then one has $\Lambda(\nu) \in \N_0$. Indeed, 
\begin{align*}
    \Lambda(\nu)=\sum_{i=1}^t \delta_i g_i(\nu)=\frac{1}{|G|}\sum_{\sigma \in G}\mathcal G_{i,\nu}(\sigma)
\end{align*}
is the average of $\mathcal G_{i,\nu}$ on $G$, so we have
\begin{align*}
    \Lambda(\nu)=n_1\left(\frac{1}{|G|}\sum_{\sigma \in G}|{\rm Fix}_{X_1}(\sigma)| \right)+\dots + n_s\left(\frac{1}{|G|}\sum_{\sigma \in G}|{\rm Fix}_{X_s}(\sigma)| \right).
\end{align*}
Each $\frac{1}{|G|}\sum_{\sigma \in G}|{\rm Fix}_{X_j}(\sigma)|$ is an integer according to Burnside's lemma.

\section{Zeta functions associated to cycle type}\label{sec_zeta function associated to cycle type}

We recall some basic terms in group theory and algebra. Consider the permutation action of $S_X$ on a finite set $X$. Any element $\sigma \in S_X$ has the cycle decomposition $\sigma=c_1c_2\cdots c_s$. Each cycle factor has finite length $L(c_j)$. We call $(L(c_1),L(c_2),\ldots,L(c_s))$ the {\it cycle type} of $\sigma$. For example, if $X=\{1,2,3,4,5\}$ and $\sigma=(1 \ 3)(2 \ 5\ 4)$, then its cycle type is $(2,3)$.

Similarly, let $P(T) \in \Q[T]$ be a polynomial, where $K$ is a field. For a finite extension $K/\Q$, $P$ factors into irreducible polynomials over $K$,
\[
    P(X)=P_1(X)\cdots P_s(X). 
\]
We say that $(\deg P_1,\ldots,\deg P_s)$ is the {\it splitting type} of $P$.

\

Now let $K_j$ be the splitting field of a polynomial $F_j$ over $\Q$, and $K$ be the compositum of $K_1,\ldots,K_s$. Hereafter, $G$ denotes the Galois group of $K$ over $\Q$, and $G_j$ stands for the Galois group of $K_j$ over $\Q$. The canonical surjection from $G$ to $G_j$ is denoted by $\pi_j$. Let $\mathcal O_j$ be the ring of integers of $K_j$, $\mathfrak p_j$ be the unique maximal ideal of $\mathcal O_j$, $\k_j:=\mathcal O_j/\mathfrak p_j$ be the residue field of $K_j$, and $\mathrm{Frob}_{\ell}$ be the Frobenius element in $G$. 
We also denote by $\Omega_j$ the set of roots of $F_j$. We have a natural embedding $G_j \hookrightarrow S_{\Omega_j}$ so that each $\sigma \in G_j$ has the cycle type $\tau$. Since $\mathrm{Gal}(\k_j/\mathbb F_{\ell})$ is generated by $\pi_j(\mathrm{Frob}_{\ell})|_{\k_j}$, if $\ell>q$ for some large $q$, then the cycle type of $\pi_j(\mathrm{Frob}_{\ell})$ is the same as the splitting type of $F_j \bmod\ell$.

Let $\mathcal T_j$ be the set of all cycle types on $\Omega_j$ (equivalently, the set of all partitions of $|\Omega_j|$) and $$\mathcal T:=\mathcal T_1\times \cdots \times \mathcal T_s,$$
the set of all joint cycle types on $(\Omega_1,\ldots,\Omega_s)$. For a given $\tau \in \mathcal T$, we define
\begin{align*}
    \mathcal Q_{\tau}:=\{\ell>q: \pi_j(\mathrm{Frob}_{\ell}) \text{ has cycle type } \tau_j \text{ for any } j=1,2,\ldots,s\}.
\end{align*}

\begin{lemma}
    The set $\mathcal Q_{\tau}$ is Frobenian of density $\delta_{\tau}:=|H_{\tau}|/|G|$, where
    \[
        H_{\tau}:=\{\sigma \in G: \pi_j(\sigma) \text{ has cycle type } \tau_j \text{ for any } j=1,2,\ldots,s\}.
    \]
\end{lemma}
\begin{proof}
    This is the classical result of Frobenius (cf. \cite{SL96}). One can directly check that $H_{\tau}$ is stable under conjugation.
\end{proof}





Let ${\mathcal C}_{\Q}$ be the $\Q$-vector space of $\Q$-valued class functions $\phi$ on $G$ such that $\phi(\sigma^t)=\phi(\sigma)$ for all $t\in \Z$ with $\gcd(t,\mathrm{ord}(\sigma))=1$. For a subgroup $G_0$ of $G$, we denote by $\chi_{G_0}^G(\mathbf 1)$ the character of an induced representation $\mathrm{Ind}_{G_0}^G(\mathbf 1)$. It is a permutation character of the action of $G$ on $G_0\backslash G$ so that
\[
    \chi_{G_0}^G(\mathbf 1)(\sigma)=|{\rm Fix}_{G_0 \backslash G}(\sigma)|.
\]
Define
\[
    \mathrm{Conj}_{cyc}(G):=\{[C]:C \text{ is a cyclic subgroup of } G\},
\]
where $[C]$ denotes the conjugacy class of $C$. The Artin induction theorem implies that every character is a rational linear combination of characters induced from the cyclic subgroups $C$ for each $[C] \in \mathrm{Conj}_{cyc}(G)$. Furthermore, if $\phi \in \mathcal C_{\Q}$, then it can be written as a rational linear combination of $\chi_C^G(\mathbf 1)$.

\begin{lemma}\label{lem_ratioanl class function basis}
    The set $\{\chi_C^G(\mathbf 1): [C] \in \mathrm{Conj}_{cyc}(G)\}$ forms a basis for ${\mathcal C}_{\Q}$.
\end{lemma}

\begin{proof}
    Note that
    \[
        \phi(\sigma^t)=\phi(\sigma) \text{ if } \gcd(t,\mathrm{ord}(\sigma))=1
    \]
    is equivalent to
    \[
        \phi(\sigma_1)=\phi(\sigma_2) \text{ if $\sigma_1$ and $\sigma_2$ generate the same cyclic subgroup.}
    \]
    By the rational conjugacy class we mean the equivalence class given by the relation $\sigma \sim g\sigma^t g^{-1}$ for $t$ with $\gcd(t,\mathrm{ord}(\sigma))=1$ and $g\in G$. Thus it is equivalent to the conjugacy class of cyclic subgroups.

    Let
    \[
        G=R_1\sqcup R_2 \sqcup \ldots \sqcup R_r
    \]
    be the partition of $G$ into the rational conjugacy classes. It is obvious that
    \[
        \{\mathbf 1_{R_1}, \mathbf 1_{R_2}, \ldots, \mathbf 1_{R_r}\}
    \]
    spans $\mathcal C_{\Q}$. On the other hand, since $r=|\mathrm{Conj}_{cyc}(G)|,$ the number of elements of $\{\chi_C^G(\mathbf 1): [C] \in \mathrm{Conj}_{cyc}(G)\}$ equals $r$. It is enough to show that this set is linearly independent. 
    
    To do this, we let $C_1,\ldots,C_r$ be representatives of the classes $[C_j] \in \mathrm{Conj}_{cyc}(G)$, ordered by $|C_1|\leq \ldots \leq |C_r|$. Let $\sigma_j$ be a generator of $C_j$ for $1\leq j \leq r$. Note that $\chi_C^G(\mathbf 1)(\sigma_j) \neq 0$ is equivalent to
    \[
        Cg\sigma_j=Cg
    \]
    for some $g\in G$, i.e., $\sigma_j$ is contained in $g^{-1}Cg$. Thus we have $|C_j| \leq |C|$ if $\chi_C^G(\mathbf 1)(\sigma_j)\neq 0$. Since $|C_j|=|C|$ implies $C_j=gCg^{-1}$, which means $[C_j]=[C]$, we may conclude that $|C_j|<|C|$ when $C=C_i$ and $\chi_C^G(\mathbf 1)(\sigma_j)\neq 0$ with $i\neq j$. In other word, if $|C_j| \geq |C_i|$, then $\chi_{C_i}^G(\sigma_j)=0$. 
    
    Suppose
    \[
        c_1\chi_{C_1}^G(\mathbf 1)+\cdots+c_r\chi_{C_r}^G(\mathbf 1)=0.
    \]
    By taking $\sigma=\sigma_r$, as $\chi_{C_i}^G(\mathbf 1)(\sigma_r)=0$ for any $i<r$, we have $c_r\chi_{C_r}^G(\sigma_r)=0$, so $c_r=0$. Now taking $\sigma=\sigma_{r-1}$, we get $c_{r-1}=0$. Consequently, we conclude that every $c_j$ is zero, which completes the proof.
\end{proof}

\begin{proposition}\label{prop_zeta function associated to cycle type is expressed as Dedekind zeta function}
    Let $\tau \in \mathcal T$ be a joint cycle type on $(\Omega_1,\ldots,\Omega_s)$. There exists a positive integer $M$ and integers $M_C$ such that
    \begin{align*}
        \zeta_{\mathcal Q_{\tau}}(w)^M=e^{\beta(w)}\prod_{[C]\in\mathrm{Conj}_{cyc}(G)}\zeta_{K^C}(w)^{M_C}
    \end{align*}
    for some holomorphic function $\beta(w)$ for $\Re(w)>1/2$, where $\zeta_{K^C}(w)$ is the Dedekind zeta function of $K^C$. Moreover, $\Re(\beta(w)) \ll_{\epsilon} 1$ for $\Re(w)\geq 1/2+\epsilon$.
\end{proposition}

\begin{proof}
Since $\mathbf 1_{H_{\tau}} \in \mathcal C_{\Q}$, by Lemma~\ref{lem_ratioanl class function basis}, there exists a positive integer $M$ and integers $M_C$ for which
\[
    M\mathbf 1_{H_{\tau}}=\sum_{[C]\in\mathrm{Conj}_{cyc}(G)} M_C \chi_C^G(\mathbf 1).
\]

Thus we have
    \begin{align*}
        M\log \zeta_{\mathcal Q_{\tau}}(w)&= M\log \zeta_{\mathbb P_{\leq q}}(w)+\sum_{p \in \mathcal Q_{\tau}}\left(\frac{M}{p^w}+\cdots\right) \\ &=M\log \zeta_{\mathbb P_{\leq q}}(w)+\sum_{p}\left(\frac{M\mathbf 1_{H_{\tau}}(\mathrm{Frob}_p)}{p^w}+\cdots\right) \\
        &=M\log \zeta_{\mathbb P_{\leq q}}(w)+\sum_{C\in\mathrm{Conj}_{cyc}(G)} \sum_p \left(\frac{M_C\chi_C^G(\mathbf 1)(\mathrm{Frob}_p)}{p^w}+\cdots\right) \\
        &=\beta(w)+\sum_{C\in\mathrm{Conj}_{cyc}(G)}\log L(s,\chi_C^G(\mathbf 1))^{M_C},
    \end{align*}
    where $\beta(w)$ is holomorphic in the region defined by $\Re(w)>1/2$, and its real part is bounded above in the region defined by $\Re(w) \geq 1/2+\epsilon$ for any $\epsilon$. Consequently, we obtain
    \begin{align*}
        \zeta_{\mathcal Q_{\tau}}(w)^M&=e^{\beta(w)}\prod_{C\in\mathrm{Conj}_{cyc}(G)}L(w,\chi_C^G(\mathbf 1))^{M_C} 
        \\
        &=e^{\beta(w)}\prod_{C\in\mathrm{Conj}_{cyc}(G)}L(w,\mathbf 1,K/K^C)(w)^{M_C} \\
        &=e^{\beta(w)}\prod_{C\in\mathrm{Conj}_{cyc}(G)}\zeta_{K^C}(w)^{M_C},
    \end{align*}
    which completes the proof.
\end{proof}

\section{Bounds for the density function}\label{sec_bounds for the density function}

Let $\rho_{f,\ell}$ be the $\ell$-adic Galois representation associated to $f$,
\begin{align*}
    \rho_{f,\ell}: \mathrm{Gal}_{\mathbb Q} \to \mathrm{GL}_2(\mathbb Q_{\ell}).
\end{align*}
The product of $\rho_{f,\ell}$ defines an adelic Galois representation 
\begin{align*}
    \rho_{f}:\mathrm{Gal}_{\mathbb Q} \to \mathrm{GL}_2(\hat{\mathbb Q}),
\end{align*}
where $\hat{\mathbb Q}$ is the ring of finite adeles. Up to conjugation (equivalence of representations), we may assume that the image of $\rho_{f,\ell}$ is contained in $\mathrm{GL}_2(\mathbb Z_{\ell})$ for each $\ell$ so that $\mathrm{im}(\rho_{f}) \subseteq \mathrm{GL}_2(\hat{\mathbb Z})$, where $\hat{\mathbb Z}:=\prod_{\ell} \mathbb Z_{\ell}$. 

For a positive integer $\delta$, the mod $\delta$ Galois representation
\begin{align*}
    \bar{\rho}_{f,\delta}:\mathrm{Gal}_{\mathbb Q} \longrightarrow \mathrm{GL}_2(\mathbb Z/\delta\mathbb Z)=\prod_{\ell^n \mid\mid \delta} \mathrm{GL}_2(\mathbb Z/\ell^n \mathbb Z)
\end{align*}
is given by the mod $\delta$-reduction of $\rho_{f}$. Thus we may write
\begin{align*}
    \bar{\rho}_{f,\delta}=\prod_{\ell^n \mid\mid \delta} \bar{\rho}_{f,\ell^n}.
\end{align*}
Recall the adelic open image theorem (\cite[Theorem~2.2.2]{Loe17}). If we let $\mathcal A_{\ell^n}'$ be the image of $\bar{\rho}_{f,\ell^n}$ and
\begin{align*}
    \mathcal A_{\ell^n}:=\{g \in \mathrm{GL}_2(\mathbb Z/\ell^n\mathbb Z):\det g=v^{k-1} \text{ for some $v \in (\mathbb Z/{\ell}^n\mathbb Z)^{\times}$}\}.
\end{align*}
then there exists $q$ such that $\mathcal A_{\ell^n}'=\mathcal A_{\ell^n}$ if $\ell>q$.

We define a subset $\mathcal C_{\delta}$ of ${\rm im}(\bar{\rho}_{f,\delta})$ by
\begin{align*}
    \mathcal C_{\delta}:=\{g\in {\rm im}(\bar{\rho}_{f,\delta}): \ F(\mathrm{tr}(g)) =0\},
\end{align*}
and the function $h:\N \to [0,1]$, defined by
\begin{align*}
    h(\delta):=\frac{|\mathcal C_{\delta}|}{|\mathrm{im}(\bar{\rho}_{f,\delta})|}.
\end{align*}
Note that if $p \nmid \delta N$, then $F(a_f(p)) \equiv 0 \pmod{\delta}$ is equivalent to  $\bar{\rho}_{f,\delta}(\mathrm{Frob}_p) \in C_{\delta}$.
 
Recall that the density function $h$ can be factorized locally as follows.
\begin{lemma}{\cite[Theorem~A.3]{LL26}}\label{lem_density function factorization}
We have
\begin{align*}
    h(\delta) \asymp_{f} \prod_{\substack{\ell^n \mid\mid \delta \\ \ell>q}} h(\ell^n).
\end{align*}
\end{lemma}
Thus it is enough to consider the local density $h(\ell^n)$ for each $\ell>q$.

We further enlarge $q$ if necessary so that every prime divisor of the leading coefficient or the discriminant of $F$ is at most $q$. For given $d \in (\mathbb Z/\ell^n)^{\times}$ and $t \in \mathbb Z/\ell^n$, we define
\begin{align*}
    A_{d,t}(\ell^n):=\{g\in \mathrm{GL}_2(\mathbb Z/\ell^n): \det(g)=d, \ \mathrm{tr}(g)=t\}.
\end{align*}
Denote $\frac{t^2}{4}-d=\ell^{v}u$, where $0\leq v\leq n$ and $1\leq u \leq \ell^{n-v}$ and $(u,\ell)=1$.

\begin{lemma}\label{lem_Ad,t number of elements}
    The number of elements of $A_{d,t}(\ell^n)$ is given by
    \begin{align*}
        |A_{d,t}(\ell^n)|=\begin{cases}
            \sum_{j=0}^{n-1}(\ell-1)\ell^{2n-\lceil j/2 \rceil-1} +  \ell^{n+\lfloor n/2 \rfloor}. & \text{if } v=n, \\
            \sum_{j=0}^{v}(\ell-1)\ell^{2n-\lceil j/2 \rceil-1}+2\ell^{2n-v/2-1} & \text{if $v<n$, $v$ is even, and $\left(\frac{u}{\ell}\right)=1$}, \\
            \sum_{j=0}^{v}(\ell-1)\ell^{2n-\lceil j/2 \rceil-1} & \text{otherwise}.
        \end{cases}
    \end{align*}
    In particular, $|A_{d,t}(\ell^n)|=\ell^{2n}+O(\ell^{2n-1})$.
\end{lemma}
\begin{proof}
    Write $g= \begin{pmatrix}x & y \\ z & t-x \end{pmatrix} \in \mathrm{GL}_2(\Z/\ell^n)$. We have to compute the number of matrices $(x,y,z) \in (\Z/\ell^n)^3$ for which $x(t-x)-yz \equiv d \pmod{\ell^n}$. This equation is equivalent to
    \begin{align*}
        \bigg(x-\frac{t}{2}\bigg)^2 \equiv \ell^v u-yz \pmod{\ell^n},
    \end{align*}
    so we let $\alpha=x-\frac{t}{2}$. We divide the cases into $v_{\ell}(y)=j=0,1,\ldots,n$, where $v_{\ell}(y)=n$ indicates $y=0$ in $\Z/\ell^n$. For a given $y$ with $v_{\ell}(y)=j$, the equation has a solution if and only if
    \begin{align}\label{eqn_R_j condition}
        \ell^j \mid \ell^v u -\alpha^2.
    \end{align}
    Further, for each such $\alpha$, the equation is equivalent to
    \begin{align}\label{eqn_z equation}
        z \equiv (y')^{-1}\ell^{-j}(\ell^v u-\alpha^2) \pmod{\ell^{n-j}},
    \end{align}
    where $y=\ell^j y'$, $y' \in (\Z/\ell^n)^{\times}$.
    There are $\ell^j$ lifts $z$ modulo $\ell^n$ satisfying \eqref{eqn_z equation}. 
    Let $R_j$ be the number of $\alpha \in \Z/\ell^n$ satisfying \eqref{eqn_R_j condition}. Note that the number of $y \in \Z/\ell^n$ with $v_{\ell}(y)=j$ is
    \begin{align*}
        \begin{cases}
            (\ell-1)\ell^{n-j-1} & \text{if} \ j<n, \\
            1 & \text{if} \ j=n.
        \end{cases}
    \end{align*}
    Thus we have
    \begin{align*}
        |A_{d,t}(\ell^n)|=\sum_{j=0}^{n-1} (\ell-1)\ell^{n-j-1}R_j \ell^j + R_n\ell^n.
    \end{align*}

    Now we compute $R_j$. When $j \leq v$, the condition \eqref{eqn_R_j condition} is equivalent to $\alpha \equiv 0 \pmod{\ell^{\lceil j/2 \rceil}}$. So the number of such $\alpha \in \Z/\ell^n$ is $\ell^{n-\lceil j/2 \rceil}$. When $j>v$, the condition \eqref{eqn_R_j condition} holds only if $v$ is even and $u$ is a square modulo $\ell$. In this case, we may write $\alpha=\ell^{v/2}\alpha'$, so the condition is equivalent to
    \begin{align*}
        \alpha'^2 \equiv u \pmod{\ell^{j-v}}.
    \end{align*}
    There are two $\alpha'$ modulo $\ell^{j-v}$ satisfying the above equation, and for each solution we have $\alpha=\ell^{v/2}\alpha'$ modulo $\ell^{j-v/2}$. Each $\alpha$ has $\ell^{n-j+v/2}$ lifts modulo $\ell^n$. In short,
    \begin{align*}
        R_j = \begin{cases}
            \ell^{n-\lceil j/2 \rceil} & \text{if} \ 0\leq j \leq v, \\
            2\ell^{n-j+v/2} & \text{if} \ j>v \ \text{and $v$ is even, $\left(\frac{u}{\ell}\right)=1$}, \\
            0 & \text{otherwise}.
        \end{cases}
    \end{align*}
    Therefore, if $v<n$ so that the exceptional case $j=n$ lies in the range $v<j$, we have
    \begin{align*}
        |A_{d,t}(\ell^n)|=\sum_{j=0}^v (\ell-1)\ell^{2n-\lceil j/2 \rceil-1}+2\delta_{v,u}\ell^{n+v/2}+2\delta_{v,u}(\ell-1)\ell^{n-1}\sum_{j=v+1}^{n-1}\ell^{n-j+v/2},
    \end{align*}
    where $\delta_{v,u}=1$ if $v$ is even and $\left(\frac{u}{\ell}\right)=1$, and $\delta_{v,u}=0$ otherwise. If $v=n$, so that no $j$ satisfies $j>v$, then 
    \begin{align*}
        |A_{d,t}(\ell^n)|=\sum_{j=0}^{n-1}(\ell-1)\ell^{2n-\lceil j/2 \rceil-1} +  \ell^{n+\lfloor n/2 \rfloor}.
    \end{align*}
    One can verify that
    \begin{align*}
        2(\ell-1)\ell^{n-1}\sum_{j=v+1}^{n-1}\ell^{n-j+v/2}=2(\ell^{2n-v/2-1}-\ell^{n+v/2}),
    \end{align*}
    which completes the proof.
\end{proof}

\begin{lemma}\label{lem_Aln number of elements}
    If $\ell \gg 1$, the number of elements of $\mathcal A_{\ell^n}$ is
    \begin{align*}
        |\mathcal A_{\ell^n}|=\frac{1}{\lambda}\ell^{4(n-1)}(\ell-1)^{2}(\ell^2+\ell),
    \end{align*}
    where
    \begin{align*}
        \lambda:=\mathrm{gcd}(\ell-1,k-1).
    \end{align*}
\end{lemma}

\begin{proof}
    This result is given in \cite[\S~3]{GM14}.
\end{proof}

Let $\Lambda_{\ell^n}:=((\Z /\ell^n \Z)^{\times})^{k-1}=\{v^{k-1} \in (\Z /\ell^n \Z)^{\times}: v \in (\Z /\ell^n \Z)^{\times}\}$. The kernel of the natural surjection
\[
    (\Z /\ell^n \Z)^{\times} \to \Lambda_{\ell^n}
\]
defined by $v \mapsto v^{k-1}$ is cyclic of order $\gcd(\ell^{n-1}(\ell-1),k-1)$, which equals $\lambda$ if $\ell>k-1$. Thus for $\ell \gg 1$, we have
\[
    |\Lambda_{\ell^n}|=\frac{\ell^n-\ell^{n-1}}{\lambda}.
\]

\begin{lemma}\label{lem_number of C ell n naive}
    If $\ell \gg 1$, the number of elements of $\mathcal C_{\ell^n}$ is given by
    \begin{align*}
        |\mathcal C_{\ell^n}|=\frac{\ell^{3n}}{\lambda}N_F(\ell)+O(\ell^{3n-1}).
    \end{align*}
\end{lemma}

\begin{proof}
    By Lemmas \ref{lem_Ad,t number of elements} and \ref{lem_Aln number of elements},
    \begin{align*}
        |\mathcal C_{\ell^n}|&=\sum_{\substack{t \in \mathbb Z/\ell^n \\ F(t)=0}} \sum_{d\in \Lambda_{\ell^n}}\sum_{\substack{g \in \mathrm{GL}_2(\mathbb Z/\ell^n) \\ \det (g)=d \\ \mathrm{tr}(g)=t}}1 \\
        &=\sum_{\substack{t \in \mathbb Z/\ell^n \\ F(t)=0}} \sum_{d\in \Lambda_{\ell^n}} (\ell^{2n}+O(\ell^{2n-1})) \\
        &=\sum_{\substack{t\in \mathbb Z/\ell^n \\ F(t)=0}}\frac{\ell^n-\ell^{n-1}}{\lambda}(\ell^{2n}+O(\ell^{2n-1})).
    \end{align*}
    Let $N_F(\ell)$ be the number of roots of $F \bmod\ell$. By Hensel's lemma, the number of $t \in \mathbb Z/\ell^n$ for which $F(t)=0$ equals $N_F(\ell)$. Therefore, we have
    \begin{align*}
        |\mathcal C_{\ell^n}|=\frac{\ell^{3n}}{\lambda}N_F(\ell)+O(\ell^{3n-1})
    \end{align*}
    as claimed.
\end{proof}

Recall that $F_i$ and $F_j$ have a common zero in $\mathbb F_{\ell}$ if and only if their resultant $\mathrm{Res}(F_i,F_j)$ vanishes modulo $\ell$. We may assume that $q$ is larger than the largest prime factor of $\mathrm{Res}(F_i, F_j)$ for all $i \neq j$ so that $F_i$ and $F_j$ have no common zero in $\mathbb F_{\ell}$ if $\ell>q$. Moreover, if $F_i$ and $F_j$ have a common zero in $\mathbb Z/\ell^n$, then they must have a common zero in $\mathbb F_{\ell}$. Therefore, we may assume $F_i$ and $F_j$ have no common zero in $\mathbb Z/\ell^n$ for all $n \in \mathbb N$ if $\ell>q$.

Thus we have
\[
    N_F(\ell)=\sum_{j=1}^s N_{F_j}(\ell)=\sum_{j=1}^s |{\rm Fix}_{\Omega_j}(\mathrm{Frob}_{\ell})|
\]
and consequently,
\begin{align*}
    h(\ell^n)=\frac{|\mathcal C_{\ell^n}|}{|\mathcal A_{\ell^n}|} &= \frac{1}{\ell^n}\sum_{j=1}^s |{\rm Fix}_{\Omega_j}(\mathrm{Frob}_{\ell})| \left(1+O\left(\frac{1}{\ell}\right)\right) \\
    &\leq \frac{1}{\ell^n}\sum_{j=1}^s |{\rm Fix}_{\Omega_j}(\mathrm{Frob}_{\ell})| \left(1+\frac{B_F}{\ell}\right)
\end{align*}
for some constant $B_F>0$.

For a joint cycle type $\tau \in \mathcal T$, note that the value $\sum_j|{\rm Fix}_{\Omega_j}(\mathrm{Frob}_{\ell})|$ is completely determined by $\tau$ for which $\ell \in \mathcal Q_{\tau}$, i.e., $\mathrm{Frob}_{\ell} \in H_{\tau}$. Thus for $\tau \in \mathcal T$, we choose any $\ell \in \mathcal Q_{\tau}$ (or $\sigma_{\tau} \in H_{\tau}$) and define
\[
    g_{\tau}:=\sum_{j=1}^s |{\rm Fix}_{\Omega_j}(\mathrm{Frob}_{\ell})|=\sum_{j=1}^s |{\rm Fix}_{\Omega_j}(\sigma_{\tau})|.
\]
We also define
\[
    \omega_{\tau}(\delta):= \text{the number of prime factors $\ell$ of $\delta$ such that $\ell \in \mathcal Q_{\tau}$.}
\]
Then
\begin{align*}
    h(\ell^n) \leq \frac{g_{\tau}}{\ell^n}\left(1+\frac{B_F}{\ell}\right) \quad \text{ if } \ell \in \mathcal Q_{\tau},
\end{align*}
and therefore, according to Lemma~\ref{lem_density function factorization},
\begin{align}\label{eqn_h(delta) upper bound}
    h(\delta)\ll_f \frac{1}{\delta}\prod_{\substack{\tau \in \mathcal T \\ g_{\tau}>0 \\ \omega_{\tau}(\delta)\geq 1}} g_{\tau}^{\omega_{\tau}(\delta)}\prod_{\substack{\ell \mid \delta \\ \ell\in \mathcal Q_{\tau}}} \left(1+\frac{B_F}{\ell}\right). 
\end{align}
The right-hand side is the function we take as the (reciprocal) `upper bound function' for the density function. Namely, we let
\begin{align}\label{eqn_def of psi_F}
    \psi_F(\delta):=\delta\prod_{\substack{\tau \in \mathcal T \\ g_{\tau}>0 \\ \omega_{\tau}(\delta)\geq 1}} g_{\tau}^{-\omega_{\tau}(\delta)} \prod_{\substack{\ell \mid \delta \\ \ell\in \mathcal Q_{\tau}}}\left(1+\frac{B_F}{\ell}\right)^{-1}.
\end{align}

Now, we will consider specific $g,\kappa,\c,\xi$ given as follows: For each joint cycle type $\tau \in \mathcal T$,
\begin{align*}
    &g(\nu)=g_{\tau}(\nu)=g_{\tau}, \quad & &\kappa(\nu)=\max_{1\leq i\leq m}\nu_i, \\
    &\c=\mathbf 0= (0,0,\ldots,0), \quad & & \xi=1.
\end{align*}

\begin{lemma}\label{lem_(g,max,0,1,Qtau) data}
    For any joint cycle type $\tau \in \mathcal T$, a quintuple $(g_{\tau},\max \nu_i,\mathbf 0,1,\mathcal Q_{\tau})$ is data where $D(\mathcal Q_{\tau})$ is given by the reciprocal density $=1/\delta_{\tau}=|G|/|H_{\tau}|$. Furthermore, we have
    \[
        \Lambda(\nu)=\sum_{\tau \in \mathcal T} \delta_{\tau}g_{\tau}=s \in \N_0.
    \]
\end{lemma}

\begin{proof}
    Conditions (i)--(vi) in Definition~\ref{def_quintuple data_irred ver} are obvious. (vii) follows from Proposition~\ref{prop_zeta function associated to cycle type is expressed as Dedekind zeta function} under GRH for Dedekind zeta functions; the function
    \[
        w\zeta_{\mathcal Q_{\tau}}(1+w)^{1/\delta_{\tau}}=w e^{\beta(1+w)/M}\prod_{[C] \in \mathrm{Conj}_{cyc}(G)}\zeta_{K^C}(1+w)^{M_C/M\delta_{\tau}}
    \]
    continues to a nonvanishing holomorphic function for $\Re(w)>-\epsilon_0$ except for $w=0$. The holomorphy and nonvanishing at $w=0$ follow from Lemma~\ref{lem_Ser Frobenian}. Moreover, this implies that
    \[
        \sum_{[C] \in \mathrm{Conj}_{cyc}(G)}\frac{M_C}{M\delta_{\tau}}=1.
    \]

    For (viii), we may choose $\epsilon_0<1/2$. Note that for every $\epsilon>0$, we have
    \begin{align*}
        \frac{w\zeta_{\mathcal Q_\tau}(1+w)^{1/\delta_\tau}}{(1+|w|)^{1+\epsilon}} \ll_{\epsilon,\epsilon} 1
    \end{align*}
    uniformly for $\Re(w) \geq -\epsilon_0$ and $|w|\leq 1$. For $\Re(w) \geq -\epsilon_0$ and $|w|> 1$, under GRH for Dedekind zeta functions, Littlewood-type bound tells us that
    \begin{align*}
        |\log \zeta_{K^C}(1+w)| = o(\log |\Im(w)|)
    \end{align*}
    (cf. \cite[Corollary~3]{Sim23}), and hence 
    \begin{align*}
        \log |\zeta_{K^C}(1+w)| = \Re(\log \zeta_{K^C}(1+w))=o(\log |\Im(w)|).
    \end{align*}
    This implies 
    \begin{align*}
        |\Im(w)|^{-\epsilon} \ll_{\epsilon_0,\epsilon} |\zeta_{K^C}(1+w)| \ll_{\epsilon_0,\epsilon} |\Im(w)|^{\epsilon}
    \end{align*}
    and consequently,
    \begin{align*}
        \zeta_{K^C}(1+w) \ll_{\epsilon,\epsilon_0} (1+|w|)^{\epsilon}, \qquad \zeta_{K^C}(1+w)^{-1} \ll_{\epsilon,\epsilon_0} (1+|w|)^{\epsilon}.
    \end{align*} 
    As $|e^{\beta(1+w)}|$ is bounded for $\Re(w) \geq -\epsilon_0,$ we obtain
    \begin{align*}
        w\zeta_{\mathcal Q_\tau}(1+w)^{1/\delta_\tau}&=w e^{\beta(1+w)/M} \prod_{[C]\in\mathrm{Conj}_{cyc}(G)} \zeta_{K^C}(1+w)^{M_C/(M\delta_\tau)}  \\
        & \ll_{\epsilon,\epsilon_0} (1+|w|)^{1+\epsilon}
    \end{align*}
    This verifies (viii) for $D(\mathcal Q_{\tau})=1/\delta_{\tau}$ and any $C(\mathcal Q_{\tau})>0$.

    Lastly, note that for each $\tau \in \mathcal T$
    \begin{align*}
        |H_{\tau}|g_{\tau}=|H_{\tau}|\sum_{j=1}^s |{\rm Fix}_{\Omega_j}(\sigma_{\tau})| =\sum_{\sigma \in H_{\tau}}\sum_{j=1}^s |{\rm Fix}_{\Omega_j}(\sigma)|. 
    \end{align*}
    Since $G=\bigsqcup_{\tau \in \mathcal T}H_{\tau}$,
    \begin{align*}
        \Lambda(\nu)=\sum_{j=1}^s\frac{1}{|G|}\sum_{\sigma \in G}|{\rm Fix}_{\Omega_j}(\sigma)|=\sum_{j=1}^s |X_i/G|=s
    \end{align*}
    according to Burnside's lemma.
\end{proof}

\begin{proposition}\label{prop_sum of psi_F^-1 bound}
    There is a constant $C_{1,F}>0$ such that
    \begin{align*}
        \sum_{d_1,\ldots,d_m \leq x}\frac{1}{ \psi_F(\lcm(d_1,\ldots,d_m))} = C_{1,F}(\log x)^{s(2^m-1)}+O((\log x)^{s(2^m-1)-1}).
    \end{align*}
\end{proposition}

\begin{proof}
    According to Theorem~\ref{thm_multivariable tauberian with many local data} and Lemma~\ref{lem_(g,max,0,1,Qtau) data}, it suffices to show that
    \[
        \varphi_{\mathcal Q_{\tau}} \in \mathcal C(g_{\tau},\max \nu_i,\mathbf 0,1,\mathcal Q_{\tau})
    \]
    for $\varphi=(\psi_F\circ \lcm)^{-1}$ and $\tau \in \mathcal T$. Although $\bigsqcup_{\tau \in \mathcal T} \mathcal Q_{\tau}=\mathbb P_{\ell>q}$, we can replace $\mathcal Q_{\tau_0}$ by $\mathcal Q_{\tau_0} \sqcup \mathbb P_{\ell \leq q}$ for one $\tau_0 \in \mathcal T$ so the primes $\ell \leq q$ are negligible.
    Suppose first $g_{\tau}>0$. If $p \in \mathcal Q_{\tau}$, then by definition~\eqref{eqn_varphi_Q} and \eqref{eqn_def of psi_F} one has
    \begin{align*}
        \varphi_{\mathcal Q_{\tau}}(p^{\nu_1},\ldots,p^{\nu_m})-g_{\tau}p^{\langle \mathbf 0,\nu \rangle-\max \nu_i} &= \varphi(p^{\nu_1},\ldots,p^{\nu_m})-g_{\tau}p^{-\max \nu_i} \\
        &=p^{-\max \nu_i}g_{\tau}\left(1+\frac{B_F}{p}\right)-g_{\tau}p^{-\max \nu_i} \\
        &=p^{-\max \nu_i-1}B_F \\
        &\ll e^{\epsilon|\nu|_1}p^{-\max \nu_i-1}.
    \end{align*}
    If $p \not\in \mathcal Q_{\tau}$, then $\varphi_{\mathcal Q_{\tau}}(p^{\nu_1},\ldots,p^{\nu_m})=0$ by definition~\eqref{eqn_varphi_Q}.

    If $g_{\tau}=0$, one can easily verify the same conditions hold.
\end{proof}

\begin{proposition}
    There is a constant $C_{2,F}>0$ such that
    \begin{align*}
        \sum_{d_1, \ldots, d_m \leq x } \frac{\lcm(d_1,\ldots,d_m)}{\psi_F(\lcm(d_1,\ldots,d_m))^2 \mathcal G_F(\lcm(d_1,\ldots,d_m))} = C_{2,F}(\log x)^{s(2^m-1)}+O((\log x)^{s(2^m-1)-1}),
    \end{align*}
    where
    \[
        \mathcal G_F(\delta):=\prod_{\substack{\tau \in \mathcal T \\ \omega_{\tau}(\delta)\geq 1}}g_{\tau}^{\omega_{\tau}(\delta)}.
    \]
\end{proposition}

\begin{proof}
    The proof is similar to the proof of Proposition~\ref{prop_sum of psi_F^-1 bound}. We let
    \[
        \varphi(d_1,\ldots,d_m)=\frac{\lcm(d_1,\ldots,d_m)}{\psi_F(\lcm(d_1,\ldots,d_m))^2 \mathcal G_F(\lcm(d_1,\ldots,d_m))}.
    \]
    If $p \in \mathcal Q_{\tau}$, then
    \begin{align*}
        \varphi_{\mathcal Q_{\tau}}(p^{\nu_1},\ldots,p^{\nu_m})-g_{\tau}p^{\langle \mathbf 0,\nu \rangle-\max \nu_i} &= \varphi(p^{\nu_1},\ldots,p^{\nu_m})-g_{\tau}p^{-\max \nu_i} \\ 
        &=\frac{p^{\max \nu_i}}{p^{2\max \nu_i}g_{\tau}^{-2}\left(1+\frac{B_F}{p}\right)^{-2}g_{\tau}}-g_{\tau}p^{-\max \nu_i} \\
        &=g_{\tau}p^{-\max \nu_i}\left(\frac{2B_F}{p}+\frac{B_F^2}{p^2}\right) \\
        &\ll e^{\epsilon|\nu|_1}p^{-\max \nu_i-1}
    \end{align*}
    and $\varphi_{\mathcal Q_{\tau}}(p^\nu_1,\ldots,\nu_m)=0$ if $p \not\in \mathcal Q_{\tau}$.
\end{proof}

\begin{proposition}
     Let $c \in (0,1)$. If $4/c\leq u < \log x/(2\log\log x)$. We have
     \begin{align*}
         \sum_{\substack{d_1,\ldots,d_m \in S(x^c,x^{1/u}) \\ d_1,\ldots,d_m \geq x^{c/4}}} \frac{\mathcal G_F(\lcm(d_1,\ldots,d_m))}{\lcm(d_1,\ldots,d_m)} \ll (\log x)^{s(2^m-1)}\exp\left(-\frac{cu\log u}{16}\right).
     \end{align*}
     If $u>\log x/(2\log\log x)$, then
     \begin{align*}
         \sum_{\substack{d_1,\ldots,d_m \in S(x^c,x^{1/u}) \\ d_1,\ldots,d_m \geq x^{c/4}}} \frac{\mathcal G_F(\lcm(d_1,\ldots,d_m))}{\lcm(d_1,\ldots,d_m)} \ll x^{-\eta}
     \end{align*}
     for some $\eta=\eta(c)>0$.
\end{proposition}

\begin{proof}
    We first consider the case when $4/c\leq u < \log x/(2\log\log x)$. Note that
    \begin{align*}
        \sum_{\substack{d_1,\ldots,d_m \in S(x^c,x^{1/u}) \\ d_1,\ldots,d_m \geq x^{c/4}}} \frac{\mathcal G_F(\lcm(d_1,\ldots,d_m))}{\lcm(d_1,\ldots,d_m)}\leq \sum_{\substack{\delta \in S(x^{mc},x^{1/u}) \\ \delta \geq x^{c/4}}} \frac{\mathcal G (\delta)\Phi_m(\delta;x^c,x^{c/4})}{\delta},
    \end{align*}
    where $\Phi_m(\delta;X,Z)$ is the number of $m$-tuples $(d_1,\ldots,d_m) \in \mathbb N^m$ such that $Z \leq d_1, \ldots, d_m \leq X$ and $\lcm(d_1,\ldots,d_m)=\delta$. For $\Phi_m(\delta):=\Phi_m(\delta;\infty,0)$ the number of $(d_1,\ldots,d_m) \in \mathbb N^m$ with $\lcm(d_1,\ldots,d_m)=\delta$, we have
    \[
        \Phi_m(\delta;X,Z)\leq \Phi_m(\delta),
    \]
    so
    \begin{align*}
        \sum_{\substack{\delta \in S(x^{mc},x^{1/u}) \\ \delta \geq x^{c/4}}} \frac{\mathcal G (\delta)\Phi_m(\delta;x^c,x^{c/4})}{\delta} \leq \sum_{\substack{\delta \in S(x^{mc},x^{1/u}) \\ \delta \geq x^{c/4}}} \frac{\mathcal G (\delta)\Phi_m(\delta)}{\delta}.
    \end{align*}
    Note that $\mathcal G\Phi_m$ is a multiplicative function and 
    \[
        \mathcal G(p^n)\Phi(p^n)=g_{\tau}((n+1)^m-n^m), \quad \text{if } p \in \mathcal Q_{\tau}. 
    \]
    Applying \cite[Lemma~3.2]{LL26}, we obtain
    \begin{align*}
        \sum_{\substack{\delta \in S(x^{mc},x^{1/u}) \\ \delta \geq x^{c/4}}} \frac{\mathcal G (\delta)\Phi_m(\delta)}{\delta}&=\exp\bigg(\sum_{p\leq x^{1/u}}\frac{\mathcal G(p)\Phi_m(p)}{p}-\frac{cu\log u}{16}\bigg) \\
        &=\exp\bigg(\sum_{\tau \in \mathcal T}\sum_{\substack{p\leq x^{1/u} \\ p \in \mathcal Q_{\tau}}}\frac{g_{\tau}(2^m-1)}{p}-\frac{cu\log u}{16}\bigg) \\
        &=\prod_{\tau \in \mathcal T}\exp\bigg(\sum_{\substack{p\leq x^{1/u} \\ p \in \mathcal Q_{\tau}}}\frac{g_{\tau}(2^m-1)}{p}\bigg)\exp\bigg(-\frac{cu\log u}{16}\bigg).
    \end{align*}
    By virtue of Serre \cite{Ser75}, we have
    \begin{align*}
        \sum_{\substack{p\leq x^{1/u} \\ p \in \mathcal Q_{\tau}}}\frac{g_{\tau}(2^m-1)}{p}=\delta_{\tau}g_{\tau}(2^m-1)(\log\log x-\log u)+O(1),
    \end{align*}
    so by Lemma~\ref{lem_(g,max,0,1,Qtau) data},
    \begin{align*}
        \prod_{\tau \in \mathcal T}\exp\bigg(\sum_{\substack{p\leq x^{1/u} \\ p \in \mathcal Q_{\tau}}}\frac{g_{\tau}(2^m-1)}{p}\bigg)\times \exp\bigg(-\frac{cu\log u}{16}\bigg)&\ll\prod_{\tau \in \mathcal T}(\log x)^{\delta_{\tau}g_{\tau}(2^m-1)}\times\exp\bigg(-\frac{cu\log u}{16}\bigg) \\
        &=(\log x)^{s(2^m-1)}\exp\bigg(-\frac{cu\log u}{16}\bigg).
    \end{align*}

    Now suppose $u>\log x/(2\log\log x)$. As above, the sum is bounded above by
    \begin{align*}
        \sum_{\substack{\delta \in S(x^{mc},x^{1/u}) \\ \delta \geq x^{c/4}}} \frac{\mathcal G (\delta)\Phi_m(\delta)}{\delta}
    \end{align*}
    and by the Cauchy--Schwarz inequality, we have
    \begin{align*}
        \bigg(\sum_{\substack{\delta \in S(x^{mc},x^{1/u}) \\ \delta \geq x^{c/4}}} \frac{\mathcal G (\delta)\Phi_m(\delta)}{\delta}\bigg)^2 \leq \sum_{\substack{\delta \in S(x^{mc},x^{1/u}) \\ \delta \geq x^{c/4}}} \frac{(\mathcal G (\delta)\Phi_m(\delta))^2}{\delta}\sum_{\substack{\delta \in S(x^{mc},x^{1/u}) \\ \delta \geq x^{c/4}}} \frac{1}{\delta}.
    \end{align*}
    The first factor is bounded as
    \begin{align*}
        \sum_{\substack{\delta \in S(x^{mc},x^{1/u}) \\ \delta \geq x^{c/4}}} \frac{(\mathcal G (\delta)\Phi_m(\delta))^2}{\delta} 
        &\ll \prod_{p \leq x^{1/u}}\bigg(1+\frac{(\mathcal G(p)\Phi_m(p))^2}{p}+\frac{(\mathcal G(p^2)\Phi_m(p^2))^2}{p^2}+\cdots\bigg) \\
        &=\prod_{p \leq x^{1/u}} \bigg(1+\frac{C(2^m-1)^2}{p}+O\bigg(\frac{1}{p^2}\bigg)\bigg),
    \end{align*}
    where $C$ is a constant larger than any $g_{\tau}^2$ for a joint cycle type $\tau \in \mathcal T$, and 
    \begin{align*}
        \prod_{p \leq x^{1/u}} \bigg(1+\frac{C(2^m-1)^2}{p}+O\bigg(\frac{1}{p^2}\bigg)\bigg) &=\exp\bigg(\sum_{p \leq x^{1/u}}\log\bigg(1+\frac{C(2^m-1)^2}{p}+O\bigg(\frac{1}{p^2}\bigg)\bigg)\bigg) \\
        &=\exp\bigg(\sum_{p \leq x^{1/u}}\frac{C(2^m-1)^2}{p}+O(1)\bigg) \\
        &\ll (\log x)^{C(2^m-1)^2}.
    \end{align*}
    On the other hand, \cite[Lemma~3.2]{LL26} asserts that the second factor is bounded by 
    \begin{align*}
        \sum_{\substack{\delta \in S(x^{mc},x^{1/u}) \\ \delta \geq x^{c/4}}} \frac{1}{\delta} \ll x^{-\eta}.
    \end{align*}
    This completes the proof.
\end{proof}

\begin{lemma}
    We have
    \begin{align*}
        \frac{\delta}{\psi_F(\delta)\mathcal G_F(\delta)} \ll (\log\log \delta)^{B_F}.
    \end{align*}
\end{lemma}

\begin{proof}
    This can be proven by the same argument as in \cite{RS62}. Let $\vartheta(x)$ be the logarithm of the product of all primes smaller than or equal to $x$. Let $\ell_1<\ldots<\ell_r$ be the distinct prime divisors of $\delta$ and $p_i$ be the $i$th prime number in increasing order. We first note that, if $\delta<\exp \vartheta(p_{\beta+1})$, then
    \begin{align*}
        \exp \vartheta(p_r)=p_1 \cdots p_r \leq \ell_1 \cdots \ell_r \leq \delta < \exp \vartheta(p_{\beta+1}),
    \end{align*}
    which implies $r \leq \beta$ and consequently,
    \begin{align*}
        \frac{\delta}{\psi_F(\delta)\mathcal G_F(\delta)}=\prod_{\substack{\tau \in \mathcal T \\ \omega_{\tau}(\delta)\geq 1}} \prod_{\ell \in \mathcal Q_{\tau}}\left(1+\frac{B_F}{\ell}\right)=\prod_{\substack{\ell>q \\ \ell \mid \delta}}\left(1+\frac{B_F}{\ell}\right) \leq \prod_{p\leq p_{\beta}}\left(1+\frac{B_F}{p}\right).
    \end{align*}
    Therefore, if there exists a positive integer $x\geq 5$ such that $\delta < \exp \vartheta(x)$ and if we denote by $\beta$ a positive integer such that $p_{\beta+1}\leq x_0<p_{\beta+2}$, then one has $\exp \vartheta(x)=\exp \vartheta(p_{\beta+1})$, which implies
    \begin{align*}
         \frac{\delta}{\psi_F(\delta)\mathcal G_F(\delta)} &\leq \prod_{p\leq x}\left(1+\frac{B_F}{p}\right) =\exp\left(\log \prod_{p\leq x}\left(1+\frac{B_F}{p}\right)\right) \ll \exp\left(\sum_{p \leq x} \log \left(1+\frac{B_F}{p}\right) \right) \ll (\log x)^{B_F}.
    \end{align*}
    Let us choose $x=\log \delta+0.9\log \delta/\log\log \delta$. According to \cite[(3.14)]{RS62}, for $\delta \geq e^{1340}$, we have
    \begin{align*}
        \log \delta < \vartheta (x),
    \end{align*}
    so we have
    \begin{align*}
        \frac{\delta}{\psi_F(\delta)\mathcal G_F(\delta)} \ll \left(\log \left(\log \delta+\frac{0.9\log \delta}{\log\log \delta}\right)\right)^{B_F} \ll (\log\log \delta)^{B_F}.
    \end{align*}
    Thus we obtain the desired result.
\end{proof}

In conclusion, $\psi=\psi_F$ and $\mathcal G=\mathcal G_F$ satisfy (U2-1)--(U2-4). 

To get the lower bound, we again use Lemma~\ref{lem_density function factorization}. We take $\mathbb N_q$ to be the multiplicative semigroup generated by the primes $\ell>q$ for which $N_F(\ell)>0$. Then we obtain, for some constant $B_f'>0$,
\begin{align}\label{eqn_lower bound for h(delta)}
    h(\delta) \gg_f \prod_{\substack{\ell^n \mid\mid \delta \\ \ell>q}}h(\ell^n) \geq\prod_{\substack{\ell^n \mid\mid \delta \\ \ell>q}} \frac{g_{\tau_{\ell}}}{\ell^n}\bigg(1-\frac{B_F'}{\ell}\bigg)=\frac{1}{\delta}\prod_{\substack{\tau \in \mathcal T \\ \omega_{\tau}(\delta)\geq 1}} g_{\tau}^{\omega_{\tau}(\delta)}\prod_{\ell \mid \delta}\bigg(1-\frac{B_F'}{\ell}\bigg)=\frac{\mathcal G_F(\delta)}{\delta}\prod_{\ell \mid \delta}\bigg(1-\frac{B_F'}{\ell}\bigg).
\end{align}
Let $$\varrho_F(\delta):=\frac{\delta}{\mathcal G_F(\delta)}\prod_{\ell\mid\delta} \bigg(1-\frac{B_F'}{\ell}\bigg)^{-1}.$$

\begin{lemma}\label{lem_(L-1) condition for main result}
    We have
    \begin{align*}
        \sum_{\substack{d_1,\ldots,d_m \leq x \\ \lcm(d_1,\ldots,d_m) \in \N_q}} \frac{1}{\varrho_F(\lcm(d_1,\ldots,d_m))} = C_{3,F}(\log x)^{s(2^m-1)}+O((\log x)^{s(2^m-1)-1}).
    \end{align*}
\end{lemma}

\begin{proof}
    The proof is similar to the proof of Proposition~\ref{prop_sum of psi_F^-1 bound}. Let
    \begin{align*}
        \varphi(d_1,\ldots,d_m)=\frac{1}{\varrho_F(\lcm(d_1,\ldots,d_m))}.
    \end{align*}
    If $p \in \mathcal Q_{\tau}$ and $\max \nu_i \geq 1$, then
    \begin{align*}
        \varphi_{\mathcal Q_{\tau}}(p^{\nu_1},\ldots,p^{\nu_m})-g_{\tau} p^{-\max \nu_i} &= g_{\tau} p^{-\max \nu_i}
        \left(1-\frac{B_F'}{p}\right)-g_{\tau} p^{-\max\nu_i} \\
        &= -B_F' g_{\tau} p^{-\max \nu_i-1}\\
        & \ll e^{\epsilon |\nu|_1}p^{-\max \nu_i-1}.
    \end{align*}
    Thus the additional factor $1-B_F'/p$ does not change the required local condition. The remaining argument is the same as in the proof of Proposition~\ref{prop_sum of psi_F^-1 bound}.
\end{proof}

\section{Effective Chebotarev--Sato--Tate law and Proof of Theorem~\ref{thm_main_polynomial_moments}}\label{sec_proof of main thm}

The main object in the study of the divisor moment is the set
\begin{align*}
    \Pi(x,I;F,f;\delta) :=  \{p\le x: 0\neq F(a_f(p))\equiv 0 \ ({\rm mod} \ \delta), \ \theta_f(p)\in I\}
\end{align*}
and its density. 
Let 
\begin{align*}
    \pi^*(x,I;F,f;\delta)&:=\# \Pi^*(x,I;F,f;\delta) \\
    &=\#\{p \leq x: p \nmid N, \ F(a_f(p)) \equiv 0 \ ({\rm mod} \ \delta), \ \theta_f(p) \in I \}
\end{align*}
and
\begin{align*}
    \pi(x,I;F,f;\delta)&:=\# \Pi(x,I;F,f;\delta).
\end{align*}

\begin{lemma}\label{lem_pi(x,[0,pi],F,f,delta) formula}
    Assuming $GRH$, we have
    \begin{align}\label{eqn_pi*(x,[0,pi];F,f;delta)}
        \pi^*(x,[0,\pi];F,f;\delta)=h(\delta)\pi(x)+O(h(\delta)x^{1/2}\log(\Delta_{K_{\delta}} x^{[K_{\delta}:\Q]})),
    \end{align}
    where $\pi(x)$ is the number of primes less than or equal to $x$, $K_{\delta}$ is a number field defined by $\ker \bar \rho_{f,\delta}=\mathrm{Gal}(\overline{\Q}/K_{\delta})$, and $\Delta_K$ is the absolute value of the discriminant of $K$.
    If $1\leq \delta \leq x^{\frac{1}{8}-\epsilon}$,
    \begin{align}\label{eqn_pi(x,[0,pi];F,f;delta)}
        \pi(x,[0,\pi];F,f;\delta)=h(\delta)\pi(x)+O\bigg(\frac{x^{7/8}}{\sqrt{\log x}}\bigg).
    \end{align}
\end{lemma}

\begin{proof}
    Let $G_{\delta}:={\rm im} (\bar \rho_{f,\delta})=\mathrm{Gal}(K_{\delta}/\Q)$. Note that $\mathcal C_{\delta}$ is stable under conjugation in $G_{\delta}$. Namely, it is a finite disjoint union of conjugacy classes in $G_{\delta}$. We apply the effective Chebotarev density theorem \cite[Theorem~1.1]{LO77} to get \eqref{eqn_pi*(x,[0,pi];F,f;delta)}.

    We have (cf. \cite[p.73]{MM84}) 
    \begin{align*}
        \log\Delta_{K_{\delta}} \leq [K_{\delta}:\Q]\log(\delta N [K_{\delta}:\Q]), \quad [K_{\delta}:\Q] \leq \delta^4
    \end{align*}
    and
    \begin{align*}
        h(\delta) \ll \frac{1}{\delta^{1-\epsilon}},
    \end{align*}
    which easily follows from \eqref{eqn_h(delta) upper bound}.

    Therefore,
    \begin{align}\label{eqn_pi^* more precise bound}
        \pi^*(x,[0,\pi];F,f;\delta)-h(\delta)\pi(x) \ll \frac{1}{\delta^{1-\epsilon}}x^{1/2}\delta^4\left(\log(\delta^5 N)+\log x\right).
    \end{align}
    To obtain \eqref{eqn_pi(x,[0,pi];F,f;delta)}, we use 
    \begin{align}\label{eqn_pistar}
        \pi^*(x,[0,\pi];F,f;\delta)-\pi(x,[0,\pi];F,f;\delta)=O\bigg(\frac{x^{7/8}}{\sqrt{\log x}}\bigg).
    \end{align}
    Indeed, if $F(a_f(p))=0$, then $a_f(p)=t$ for some $t \in \C$, whereas the number of $p\leq x$ with $a_f(p)=t$ for a given $t$ is $\ll x^{7/8}/\sqrt{\log x}$ by \cite{Ser81}.
    
    Combining \eqref{eqn_pi^* more precise bound} and \eqref{eqn_pistar}, we get \eqref{eqn_pi(x,[0,pi];F,f;delta)}.
\end{proof}

\begin{lemma}\label{lem_pi(x,I,F,f,delta) formula}
    Assuming GRH and Hypothesis~\ref{hypo_nice analytic properties of L-function}, we have
    \begin{align*}
        \pi^*(x,I;F,f;\delta)=h(\delta)\mu_{ST}(I)\pi(x)+O(x^{3/4}(\log x)^{1/2}h(\delta)^{1/2}\delta^2)+O(x^{3/4}h(\delta)^{1/2}\delta^4(\log(\delta N))^{1/2}).
    \end{align*}
    If $1\leq \delta \leq x^{\frac{1}{28}-\epsilon}$, then
    \begin{align*}
        \pi(x,I;F,f;\delta)=h(\delta)\mu_{ST}(I)\pi(x)+O\bigg(\frac{x^{7/8}}{\sqrt{\log x}}\bigg).
    \end{align*}
\end{lemma}

\begin{proof}
    According to \cite[Theorem~6.4]{LL26}, we have
    \begin{align*}
        \pi^*(x,I;F,f;\delta)-h(\delta)\mu_{ST}(I)\pi(x) \ll x^{3/4}|C_{\delta}|^{1/2}(\log(x\mathcal M(K_{\delta}/\Q)N))^{1/2}.
    \end{align*}
    Since 
    \begin{align*}
        |C_{\delta}|\leq h(\delta)[K_{\delta}:\Q]\leq \delta^4 h(\delta)
    \end{align*}
    and
    \begin{align*}
        \log (\mathcal M(K_{\delta}/\Q)N) \ll \delta^5 \log(\delta N)
    \end{align*}
    (cf. \cite[\S~6.2]{LL26}), we have
    \begin{align*}
        x^{3/4}|C_{\delta}|^{1/2}(\log(x\mathcal M(K_{\delta}/\Q)N))^{1/2} \ll x^{3/4}h(\delta)^{1/2}\delta^2(\log x+\delta^4 \log(\delta N))^{1/2}.
    \end{align*}
    The remaining part follows by the same argument as in the proof of Lemma~\ref{lem_pi(x,[0,pi],F,f,delta) formula}.
\end{proof}

We now prove our main result.

\begin{proof}[Proof of Theorem~\ref{thm_main_polynomial_moments}]
    Set
    \begin{align*}
        &A_x:=\{p \leq x: p\nmid N, \ \theta_f(p) \in I, \ F(a_f(p))\neq 0\},  \\
        &n(p):=F(a_f(p)).
    \end{align*}
    Then
    \begin{align*}
        \pi(x,\delta)=\pi(x,I;F,f;\delta).
    \end{align*}
    We also put $\psi=\psi_F$, $\mathcal G=\mathcal G_F$, and $\varrho=\varrho_F$. The conditions (U2-1)--(U2-3) and (L-1) for $\psi$ and $\varrho$ have been verified in \S~\ref{sec_bounds for the density function}. The condition (U1) for $n(a)$ follows from the Deligne bound, and (U2) and (L) for $\pi(x,\delta)$ with $c_0=c_1=\frac{1}{28}-\epsilon$ were established in Lemma~\ref{lem_pi(x,I,F,f,delta) formula} (if $I=[0,\pi]$, then one may choose $c_0=c_1=\frac{1}{8}-\epsilon$ by Lemma~\ref{lem_pi(x,[0,pi],F,f,delta) formula}.) The assertion follows from Theorems~\ref{thm_auxiliary upper} and \ref{thm_auxiliary lower}.
\end{proof}

\end{document}